\documentclass[psamsfonts]{amsart}
\usepackage{amssymb,amsfonts,dsfont}
\usepackage[usenames,dvipsnames]{color}
\usepackage[shortalphabetic]{amsrefs}
\usepackage[all,arc,2cell]{xy}
\UseAllTwocells
\usepackage{enumerate}
\usepackage{ulem}
\usepackage{mathrsfs}
\usepackage[usenames,dvipsnames]{color}

\usepackage{todonotes}

\usepackage{hyperref}  
\hypersetup{%
  bookmarksnumbered=true,%
  bookmarks=true,%
  colorlinks=true,%
  linkcolor=blue,%
  citecolor=blue,%
  filecolor=blue,%
  menucolor=blue,%
  pagecolor=blue,%
  urlcolor=blue,%
  pdfnewwindow=true,%
  pdfstartview=FitBH}

\def\makeautorefname#1#2{\expandafter\def\csname#1autorefname\endcsname{#2}}
\def\equationautorefname~#1\null{(#1)\null}
\makeautorefname{footnote}{footnote}%
\makeautorefname{item}{item}%
\makeautorefname{figure}{Figure}%
\makeautorefname{table}{Table}%
\makeautorefname{part}{Part}%
\makeautorefname{appendix}{Appendix}%
\makeautorefname{chapter}{Chapter}%
\makeautorefname{section}{Section}%
\makeautorefname{subsection}{Section}%
\makeautorefname{subsubsection}{Section}%
\makeautorefname{paragraph}{Paragraph}%
\makeautorefname{subparagraph}{Paragraph}%
\makeautorefname{theorem}{Theorem}%
\makeautorefname{thm}{Theorem}%
\makeautorefname{mainthm}{Main theorem}%
\makeautorefname{introthm}{Theorem}%
\makeautorefname{introrem}{Remark}%
\makeautorefname{corollary}{Corollary}%
\makeautorefname{cor}{Corollary}%
\makeautorefname{lemma}{Lemma}%
\makeautorefname{lem}{Lemma}%
\makeautorefname{sublemma}{Sublemma}%
\makeautorefname{sublem}{Sublemma}%
\makeautorefname{subl}{Sublemma}%
\makeautorefname{prop}{Proposition}%
\makeautorefname{property}{Property}
\makeautorefname{pro}{Property}
\makeautorefname{step}{Step}%
\makeautorefname{conject}{Conjecture}%
\makeautorefname{conj}{Conjecture}%
\makeautorefname{questn}{Question}
\makeautorefname{quest}{Question}
\makeautorefname{qn}{Question}
\makeautorefname{definition}{Definition}%
\makeautorefname{defn}{Definition}%
\makeautorefname{defi}{Definition}%
\makeautorefname{def}{Definition}%
\makeautorefname{dfn}{Definition}%
\makeautorefname{rem}{Remark}%
\makeautorefname{rems}{Remarks}%
\makeautorefname{rmk}{Remark}%
\makeautorefname{rk}{Remark}%
\makeautorefname{remarks}{Remarks}%
\makeautorefname{rems}{Remarks}%
\makeautorefname{rmks}{Remarks}%
\makeautorefname{rks}{Remarks}%
\makeautorefname{axiom}{Axiom}%
\makeautorefname{axi}{Axiom}%
\makeautorefname{ax}{Axiom}%
\makeautorefname{case}{Case}%
\makeautorefname{claim}{Claim}%
\makeautorefname{clm}{Claim}%
\makeautorefname{conclusion}{Conclusion}%
\makeautorefname{concl}{Conclusion}%
\makeautorefname{conc}{Conclusion}%
\makeautorefname{cond}{Condition}%
\makeautorefname{const}{Construction}%
\makeautorefname{con}{Construction}%
\makeautorefname{criterion}{Criterion}%
\makeautorefname{criter}{Criterion}%
\makeautorefname{crit}{Criterion}%
\makeautorefname{problem}{Problem}%
\makeautorefname{problm}{Problem}%
\makeautorefname{prob}{Problem}%
\makeautorefname{prob}{Problem}%
\makeautorefname{soln}{Solution}%
\makeautorefname{sol}{Solution}%
\makeautorefname{sum}{Summary}%
\makeautorefname{oper}{Operation}%
\makeautorefname{obs}{Observation}%
\makeautorefname{ob}{Observation}%
\makeautorefname{conv}{Convention}%
\makeautorefname{cvn}{Convention}%
\makeautorefname{note}{Note}%
\makeautorefname{fact}{Fact}%
\theoremstyle{plain}
\newtheorem{thm}{Theorem}[section]
\newtheorem{cor}{Corollary}[section]
\newtheorem{prop}{Proposition}[section]

\theoremstyle{definition}
\newtheorem{defn}{Definition}[section]
\newtheorem{con}{Construction}[section]
\newtheorem{notn}{Notation}[section]
\newtheorem{rem}{Remark}[section]

\makeatletter
\let\c@cor=\c@thm
\let\c@prop=\c@thm
\let\c@lem=\c@thm
\let\c@conj=\c@thm
\let\c@defn=\c@thm
\let\c@notn=\c@thm
\let\c@rem=\c@thm
\let\c@con=\c@thm
\let\c@equation\c@thm
\numberwithin{equation}{section}
\makeatother

\makeatletter
\let\c@equation\c@thm
\makeatother
\numberwithin{equation}{section}

\newcommand{\bm}{\mathbf{m}}
\newcommand{\bn}{\mathbf{n}}
\newcommand{\bA}{\mathbb{A}}
\newcommand{\cA}{\mathcal{A}}
\newcommand{\bB}{\mathbb{B}}
\newcommand{\bC}{\mathbb{C}}
\newcommand{\bD}{\mathbb{D}}
\newcommand{\bF}{\mathbb{F}}
\newcommand{\bJ}{\mathbb{J}}
\newcommand{\bK}{\mathbb{K}}
\newcommand{\bL}{\mathbb{L}}
\newcommand{\bN}{\mathbb{N}}
\newcommand{\bO}{\mathbb{O}}
\newcommand{\bP}{\mathbb{P}}
\newcommand{\bQ}{\mathbb{Q}}
\newcommand{\bR}{\mathbb{R}}
\newcommand{\bT}{\mathbb{T}}
\newcommand{\bU}{\mathbb{U}}
\newcommand{\sA}{\mathscr{A}}
\newcommand{\sB}{\mathscr{B}}
\newcommand{\sC}{\mathscr{C}}
\newcommand{\sD}{\mathscr{D}}
\newcommand{\sE}{\mathscr{E}}
\newcommand{\sI}{\mathscr{I}}
\newcommand{\sL}{\mathscr{L}}
\newcommand{\sM}{\mathscr{M}}
\newcommand{\sO}{\mathscr{O}}
\newcommand{\sP}{\mathscr{P}}
\newcommand{\sS}{\mathscr{S}}
\newcommand{\sT}{\mathscr{T}}
\newcommand{\sV}{\mathscr{V}}
\newcommand{\sZ}{\mathscr{Z}}
\newcommand{\al}{\alpha}
\newcommand{\be}{\beta}
\newcommand{\ga}{\gamma}
\newcommand{\de}{\delta}
\newcommand{\epz}{\varepsilon}
\newcommand{\ze}{\zeta}
\newcommand{\et}{\eta}
\newcommand{\tha}{\theta}
\newcommand{\io}{\iota}
\newcommand{\la}{\lambda}
\newcommand{\rh}{\rho}
\newcommand{\si}{\sigma}
\newcommand{\ta}{\tau}
\newcommand{\om}{\omega}
\newcommand{\GA}{\Gamma}
\newcommand{\DE}{\Delta}
\newcommand{\LA}{\Lambda}
\newcommand{\PI}{\Pi}
\newcommand{\SI}{\Sigma}
\newcommand{\OM}{\Omega}

\newcommand{\ub}[1]{\underline{#1}}

\newcommand{\ltarr}{\longleftarrow}
\newcommand{\rtarr}{\longrightarrow}
\newcommand{\com}{\circ}     
\newcommand{\iso}{\cong}     
\newcommand{\htp}{\simeq}    
\newcommand{\sma}{\wedge}    
\newcommand{\wed}{\vee}      
\newcommand{\id}{\mathrm{id}}
\newcommand{\oid}{\mathds{1}} 
\newcommand{\im}{\mathrm{im}}
\newcommand{\tand}{\text{\ \ and \ \ }}   
\DeclareMathOperator*{\colim}{colim}

\newcommand{\Mult}{\mathbf{Mult}}
\newcommand{\Fix}{\mathrm{Fix}}

\newcommand{\SIG}{\mathbf{\SI^{\infty}_G}} 
\newcommand{\Sph}{\mathbf{S_G}}
\newcommand{\PR}{\mathbf{Pr}_G}

\newcommand{\Eone}{\sE'}
\newcommand{\GEone}{G\Eone}
\newcommand{\EGone}{\sE_G'}

\newcommand{\tG}{\mathcal{E}G}
\newcommand{\tSI}{\mathcal{E}\SI}

\newcommand{\All}{\mathrm{All}}
\newcommand{\Orb}{\mathrm{Orb}}
\newcommand{\GDAll}{G\sD_{\All}}
\newcommand{\GDZAll}{G\sD^\sZ_{\All}}
\newcommand{\GDOrb}{G\sD_{\Orb}}
\newcommand{\GDZOrb}{G\sD^\sZ_{\Orb}}
\newcommand{\DZAll}{\sD^\sZ_{\All}}

\title{Models of $G$-spectra as presheaves of spectra}

\author{Bertrand J. Guillou}
\email{bertguillou@uky.edu}
\address{Department of Mathematics,
University of Kentucky, 
Lexington, KY 40506 USA}

\author{J. Peter May}
\email{may@math.uchicago.edu}
\address{Department of Mathematics,
The University of Chicago, 
Chicago, IL 60637 USA}

\thanks{B.~J.~Guillou was  partially supported  by  Simons Collaboration Grant  No. 282316 and NSF grants DMS-1710379 and DMS-2003204.}

\begin{document}

\begin{abstract}

Let $G$ be a finite group. We give Quillen equivalent 
models for the category of $G$-spectra as categories
of spectrally enriched functors from explicitly
described domain categories to nonequivariant spectra.
Our preferred model is based on equivariant infinite loop 
space theory applied to elementary categorical data. 
It recasts equivariant stable homotopy theory in terms of
point-set level categories of $G$-spans and nonequivariant
spectra.  We also give a more topologically grounded model 
based on equivariant Atiyah duality. 
\end{abstract}

\maketitle

\tableofcontents

\section*{Introduction}

The equivariant stable homotopy category is of fundamental importance in algebraic topology. It is the natural home in which to study equivariant stable homotopy theory, a subject that has powerful and unexpected nonequivariant applications and is also of great intrinsic interest.   The foundations were well established by the mid-1980's, and by then the importance of working with equivariant spectra had already become abundantly clear, especially with Carlsson's proof of the Segal conjecture \cite{Carl0}.  The following decade saw much further progress;  Mackey functor and $RO(G)$-graded cohomology theories came of age, the Tate square and norm maps were introduced and given their first applications \cites{GrM2,GrM3}, and  THH, TC, and their applications to algebraic $K$-theory had made their appearance \cite{BHM}.  Summary accounts of where the subject stood in the mid-1990's are given in \cites{Carl, GrM1, EHCT}.   While there was continued work in the following decade, the subject really took hold in the mainstream of algebraic topology with its unexpected role in the 2009 solution of the Kervaire invariant problem by Hill, Hopkins, and Ravenel \cite{HHR}.   For example, on a foundational level, understanding norms as maps of equivariant spectra plays a key role.  

The first draft of this paper appeared in 2011, and the subject has truly blossomed in the decade since.   Formally, just as the category of $G$-spaces is Quillen equivalent to the presheaf category of contravariant functors from the orbit category of $G$ to spaces, the category of $G$-spectra is Quillen equivalent to the presheaf category of spectrally enriched contravariant functors from its full subcategory of suspension spectra of orbits to spectra. We shall say more about that shortly.  The purpose of this paper is to  replace the target presheaf category by one that is Quillen equivalent and yet is accessible to concrete constructions on the level of related presheaf categories of spaces and categories.

Setting up the equivariant stable homotopy category with its attendant 
model structures takes a fair amount of work. The first version was
due to Lewis and May \cite{LMS}, and more modern versions that we shall start from
are given in Mandell and May  \cite{MM} and, even more recently, \cite{HHR}.
A result of Schwede and Shipley \cite{SS} (reworked in \cite{GM0} to give the starting point of this paper)
asserts that any stable model category $\sM$ is equivalent to a category $\mathbf{Pre}(\sD,\sS)$ of 
{spectrally} enriched presheaves  with values in a chosen category $\sS$ of 
spectra.  However, the domain $\sS$-category $\sD$ is a \textit{full} $\sS$-subcategory of 
$\sM$ and typically is as inexplicit and mysterious as $\sM$ itself. From the point of view of 
applications and calculations, this is therefore only a starting point.   
One wants a more concrete understanding of the category $\sD$.  
We shall give explicit equivalents to the domain category $\sD$ in the case when
$\sM = G\sS$ is the category of $G$-spectra for a finite group $G$, and we fix
a finite group $G$ throughout.  

We shall define an $\sS$-category (or spectral category) $G\sA$ by applying  
a suitable infinite loop space machine to simply defined categories of finite $G$-sets.
The spectral category $G\sA$ is a spectrally enriched version of the Burnside 
category of $G$.  We shall prove the following result.

\begin{thm}[Main theorem]\label{MAINMAIN} There is a zig-zag of Quillen equivalences
\[ G\sS \simeq \mathbf{Pre}(G\sA,\sS)\]
relating the category of $G$-spectra to the category of spectrally enriched contravariant 
functors $G\sA\rtarr \sS$.
\end{thm}

{Such functors are often called presheaves.}  We reemphasize the simplicity of our spectral 
category $G\sA$: no prior knowledge of $G$-spectra is required to define it. 

We give a precise description of the relevant categorical input and restate the
main theorem more precisely in \autoref{sec:GA}.  The central point of the proof is to use equivariant infinite 
loop space theory to construct the spectral category $G\sA$ from elementary categories 
of finite $G$-sets.  We prove our main theorem in \autoref{sec:MainThm}, using the equivariant Barratt-Priddy-Quillen (BPQ) 
theorem to compare $G\sA$ to the spectral category $\GDAll$ given by the suspension $G$-spectra $\SI^{\infty}_G(A_+)$  
of based finite $G$-sets $A_+$, which is a standard choice for application of the theorem of Schwede and Shipley to $G\sS$.
The classical Burnside category of isomorphism classes of spans of finite $G$-sets leads to a calculation of the homotopy 
category $\text{Ho}\GDAll$ (see \autoref{hostart} below), and $G\sA$ starts from the bicategory of such 
spans, in which isomorphisms of spans give the $2$-cells.

Intuitively, (algebraic) Mackey functors can be viewed as functors from $\text{Ho}\GDAll$ to abelian groups, and the 
result of Schwede and Shipley says that $G$-spectra can be viewed as functors from $\GDAll$ to spectra.  We
are lifting the standard purely algebraic understanding of Mackey functors to obtain an analogous 
algebraic understanding of $G$-spectra as functors from $G\sA$ to spectra.
{Thus the slogan is that $G$-spectra are spectral Mackey functors.}

 It is crucial to our work that the $G$-spectra $\SI^{\infty}_G(A_+)$ are self-dual.  Our original 
proof took this as a special case of equivariant Atiyah duality (\autoref{Atiyah}), thinking of $A$ as a trivial 
example of a smooth closed $G$-manifold.  We later found a direct categorical proof (\autoref{selfish}) of this duality 
based on equivariant infinite loop space theory  and the equivariant BPQ theorem.  This allows us to give an
illuminating new proof of the required self-duality as we go along. We give presheaf versions of a few standard 
constructions on $G$-spectra in \autoref{sec:Comparisons}.  Switching gears, we give an alternative presheaf model for the category 
of $G$-spectra in terms of classical Atiyah duality in \autoref{SecAt}. 
An appendix, \autoref{background}, provides some background on the two model categories of $G$-spectra used here, equivariant orthogonal spectra and equivariant $S$-modules, and describes and compares the specialization of \cite{GM0} to those categories that provides 
the starting point for our work.

We take what we need from equivariant infinite loop space theory as a black box in this paper.  The additive and multiplicative space level theories are worked out in \cite{MMO} and \cite{GMMO1}, respectively.  The generalization from space level to category level input is based on general (and not necessarily equivariant) categorical coherence theory that is worked out in  \cite{GMMOMain}.   What is needed for this paper  is a small part of the full story there.

We thank a first diligent referee for demanding a reorganization of our original paper.  We thank a second diligent referee for an incredibly detailed list of sixty one well-thought through detailed suggestions for improving the exposition.
We also thank Ang\'{e}lica Osorno and Inna Zakharevich for very helpful comments, and we especially thank Osorno and Anna Marie Bohmann for catching an error in the handling of pairings in earlier versions of this work.  That error is one reason for the very long delay in the publication of this paper, which was first posted on ArXiv  on August 21, 2011. The delay is no fault of this journal.

In the interim, we teamed with Osorno and Mona Merling to fully work out the relevant infinite loop space theory, which turned out to be both surprisingly demanding and unexpectedly interesting. Also in the interim, 
Bohmann and Osorno \cite{BO} 
introduced categorical Mackey functors and used these, together with our main result, to produce a functorial construction of equivariant Eilenberg-Mac~Lane spectra for Mackey functors. The prospect of applications like theirs was a major motivation for our variant of the Schwede and Shipley model for the homotopy category of $G$-spectra.
A small error\footnote{We are grateful to Ang\'{e}lica Osorno for helping us discover and fix this error.}
 in \cite{BO}  is corrected in the short appendix, \autoref{Whisker}, of this paper.
 Further applications to the concrete construction of genuine $G$-spectra are in development 
 in their work and in work of Cary Malkievich and Merling \cites{MalkMerl,MalkMerl20}.  
 During the delay, Jonathan Rubin combed through our draft and caught a great many errors of detail  and infelicities.  Needless to say, we are responsible for all that remain. 
 
\medskip
\noindent {\bf Comparison with alternative approaches.}
We also note that since this article first appeared online in 2011, several alternative approaches have been given by other authors. First among these was the work of Barwick \cite{CB}.
A notable difference is that our spectral Burnside category $G\sA$ is a group completion of Barwick's effective Burnside category. A second difference is that Barwick is working in the $\infty$-categorical setting, so that questions of strictness, such as those necessitating our \autoref{Whisker}, do not arise.
Moreover, Barwick's work provides a conceptual  generalization that applies to handle the case of profinite groups, as well as other applications.
Later, streamlined alternative approaches were given in \cite{Nar} and \cite[Appendix~A]{CMNN}. The version described in \cite{CMNN} 
 has the advantage of providing a monoidal equivalence (see also \cite[Section~11]{BGS}). 
See \autoref{Monoidal} for further discussion.

\section{The bicategory $G\sE$ and
 $\sS$-category $G\sA$ }
\label{sec:GA}

In this paper, $\sS$ denotes the category of (nonequivariant) orthogonal spectra,
and $G\sS$ denotes the category of orthogonal $G$-spectra. For most of the paper, we index $G\sS$ on a complete universe, but in \autoref{background} we allow a more general universe.
See \autoref{background} for some discussion of the comparison between models of $G$-spectra.
We first define the $\sS$-category $G\sA$ (\autoref{KGE}) and restate our main theorem. 
Conceptually $G\sA$ can be viewed as obtained by 
applying a nonequivariant infinite loop space machine $\bK$ to a category $G\sE$
``enriched in permutative categories''.\footnote{{A permutative category is a symmetric strict monoidal category.}}
The term in quotes can be made categorically
precise \cites{Guill, HyPow, Schmitt}, but we shall use it just as an informal
slogan since no real categorical background is necessary to our work here: we shall give 
direct elementary definitions of the examples we use, and they do satisfy the 
axioms specified in the cited sources. 
We then define (\autoref{compcat}) a $G$-category\footnote{{In general, we understand a $G$-category to be a category internal and not just enriched in $G$-sets,
meaning that $G$ can act on both objects and morphisms.}} $\sE_G$ ``enriched in permutative $G$-categories'', from
which $G\sE$ is obtain by passage to $G$-fixed subcategories. 
 \autoref{duality} contains a discussion of duality that will be needed in \autoref{sec:MainThm} for the proof of our main theorem.

\subsection{The bicategory $G\sE$ of $G$-spans}\label{Gspan}

In any category $\sC$ with pullbacks, the bicategory of spans in $\sC$ has $0$-cells the objects
of $\sC$. The $1$-cells 
from $A$ to $B$ are zig-zags $\xymatrix@1{   B & D \ar[l] \ar[r] & A \\}  
$ of morphisms in $\sC$, and 2-cells between two such are diagrams
\begin{equation}\label{spans} 
\raisebox{5.8ex}{ \xymatrix@R=1em{  
 & D \ar[dl] \ar[dr]  \ar[dd]^{\iso} & \\
 B & & A. \\
 & E. \ar[ul]\ar[ur] }}
 \end{equation}
Composites of $1$-cells are given by (chosen) pullbacks
\begin{equation}\label{pbdiag} 
\xymatrix @R=1em{
&& F \ar[dl] \ar[dr] && \\
& E \ar[dl] \ar[dr] & & D \ar[dl] \ar[dr] & \\
C && B && A.  \\}
\end{equation}
The identity $1$-cells are the diagrams $\xymatrix@1{A & A \ar[l]_-{=} \ar[r]^-{=} & A}$. 
The associativity and unit constraints are determined by the universal property of pullbacks.
Observe that the $1$-cells $A\rtarr B$ can just as well
be viewed as objects over $B\times A$. 
Viewed this way, the identity $1$-cells are given
by the diagonal maps $\DE\colon A\rtarr A\times A$, and the composition can be displayed in the diagram
\begin{equation}\label{pbdiagtoo}
\xymatrix  
{ E\times D \ar[d] & &  F   \ar[ll] \ar[dr] \ar[d] &  \\
C\times B\times B \times A  & & C\times B \times A  \ar[ll]^-{\id \times \DE \times \id} \ar[r]_-{\pi} & C \times A, \\}
\end{equation}
where the square is a pullback and $\pi$ is the projection.  That is, composition is obtained from the obvious composition of maps to products by pulling back contravariantly along $\id\times \DE\times \id$  and then pushing forward covariantly along $\pi$.  See \cite{PS}*{Theorem~5.2} for an illuminating discussion of bicategories of spans from this point of view.

Our starting point is the bicategory of spans of (unbased) finite $G$-sets.  Here the
disjoint union of $G$-sets over $B\times A$ gives us a symmetric monoidal structure 
on the category of $1$-cells and $2$-cells $A\rtarr B$ for each pair $(A,B)$.  We can think of the bicategory of spans as a category ``enriched in the 
category of symmetric monoidal categories''.  Again, the notion in quotes does not make obvious mathematical sense since there is no obvious monoidal structure on the category of symmetric monoidal categories, but category theory due to the first author \cite{Guill} (see also \cites{HyPow, Schmitt}) explains what these objects are and how to rigidify them to categories enriched in permutative categories. 

We repeat that we have no need to go into such categorical detail. Rather than apply such category theory, we give a direct elementary 
construction of a strict structure that is equivalent to the intuitive notion of the category ``enriched in symmetric monoidal categories'' of 
spans of finite $G$-sets. We first define a bipermutative category $G\sE(1)$ that is equivalent to the symmmetric bimonoidal 
{groupoid}
of finite $G$-sets. 

\begin{defn}\label{Biperm} 
Any finite $G$-set is isomorphic to one of the form  $A=\ub{n}^\al$, where $\ub{n} = \{1,\cdots, n\}$, $\al$
is a homomorphism $G\rtarr \SI_n$, and $G$ acts on $\ub{n}$ by $g\cdot i = \al(g)(i)$ for $1\leq i\leq n$. 
{\em We understand finite $G$-sets to be of this restricted form from now on.} A $G$-map 
$f\colon \ub{m}^\al\rtarr \ub{n}^\be$ is a function $f\colon \ub{m}\rtarr \ub{n}$ such
that $f\com\al(g) = \be(g)\com f$ for $g\in G$. The morphisms of $G\sE(1)$ are the isomorphisms
$\ub{n}^\al\rtarr \ub{n}^\be$ of $G$-sets.  

The disjoint union $D\amalg E$ of finite $G$-sets
$D = \ub{s}^\si$ and $E = \ub{t}^\ta$ is $\ub{s+t}^{\si \oplus \ta}$, with $\si \oplus \ta$ being the
evident block sum $G\rtarr \SI_{s+t}$. 
 With the evident commutativity isomorphism, this gives the 
permutative 
groupoid\footnote{Though the terminology ``permutative category'' is more prevalent than ``permutative groupoid'', we find it useful to remind the reader that we are only considering isomorphisms.}
$G\sE(1)$ of finite $G$-sets; the empty finite $G$-set is the unit for $\amalg$.  
To define the cartesian product, for each $s$ and $t$ let $\la_{s,t}\colon \ub{st}\rtarr \ub{s}\times \ub{t}$ denote the lexicographic ordering. 
Then $D\times E$ is $\ub{st}^{\si\otimes \ta}$ 
where $\si\otimes \ta$ is the permutation
\[ \ub{st}\xrightarrow{\la_{s,t}}\ub{s}\times \ub{t} \xrightarrow{\si\times\ta} \ub{s}\times \ub{t} \xrightarrow{\la_{s,t}^{-1}} \ub{st}\]
as in \cite[(3.6)]{GMMOMain}.
There is again an evident commutativity isomorphism, and $\amalg$ and $\times$ 
give $G\sE(1)$ a structure of bipermutative category 
in the sense of \cite{MQR}; the multiplicative unit is the trivial $G$-set $1=(\ub{1},\epz)$, 
where $\epz(g) = 1$ for $g\in G$.

As we will need it later, we also introduce the reordering permutation $\tau_{s,t}\in \SI_{st}$, defined as the composition
\[ \ub{st} \xrightarrow{\la_{s,t}} \ub{s}\times \ub{t} \xrightarrow{\iso} \ub{t}\times \ub{s} \xrightarrow{\la_{t,s}^{-1}} \ub{ts} = \ub{st}.\]
as in \cite[Definition~3.8]{GMMOMain}.
\end{defn}

We may view $G\sE(1)$ as the groupoid of finite $G$-sets over the one point $G$-set $1$,
and we generalize the definition as follows.

\begin{defn}\label{PermCat}  
For a finite $G$-set $A$, we define a permutative 
{groupoid} $G\sE(A)$ of finite
$G$-sets over $A$. The objects of $G\sE(A)$ are the $G$-maps $p\colon D\rtarr A$.  The morphisms 
$p\rtarr q$, $q\colon E\rtarr A$, are the $G$-isomorphisms $f\colon D\rtarr E$ such that $q\com f=p$.  
Disjoint union of $G$-sets over $A$ gives $G\sE(A)$ a structure of permutative category; 
its unit is the empty set over $A$.  When $A=1$, $G\sE(A)$ is  the (``additive'') permutative category 
of the previous definition. 
\end{defn} 

\begin{rem}\label{product}  There is also a product 
$\times\colon G\sE(A)\times G\sE(B)\rtarr G\sE(A\times B)$. It takes $(D,E)$ to 
$D\times E$, where $D$ and $E$ are finite $G$-sets over $A$ and $B$, respectively. This product 
is also strictly associative and unital, with unit the unit of $G\sE(1)$, and it has 
an evident commutativity isomorphism. Restriction to the object $1$ gives the ``multiplicative'' 
permutative category of \autoref{Biperm}. This product distributes over $\amalg$ and {\it almost} makes the 
enriched category $G\sE$ of the next definition 
into a ``category enriched in permutative categories'', in the sense defined in \cite{Guill}. 
The ``almost''
refers to the fact that the category we define does not have a strict unit, a problem that was encountered in \cite{BO}
and is fixed in \autoref{Whisker} below.
\end{rem}

\begin{defn}\label{span1a}  
We define a bicategory $G\sE$ with a permutative {hom groupoid}
for each pair
of objects as follows.  The $0$-cells of $G\sE$ are the finite $G$-sets, which may be thought of as 
the categories $G\sE(A)$.  The permutative 
{groupoid} $G\sE(A,B)$ 
of $1$-cells and $2$-cells $A\rtarr B$ is $G\sE(B\times A)$, as defined in 
\autoref{PermCat}. The $1$-cells are thought of as spans and the $2$-cells as isomorphisms of spans. The composition 
\[ \circ\colon G\sE(B,C)\times G\sE(A,B)\rtarr G\sE(A,C) \]
is defined via pullbacks, as in the diagram (\ref{pbdiag}).  
Precisely, following \cite[7.2]{BO}, we choose the pullback $F$ in (\ref{pbdiag}) to be the sub $G$-set of $E\times D$, ordered lexicographically, consisting 
of the elements $(e,d)$ such that $d$ and $e$ map to the same element of $B$. 
The diagonal map $\Delta_A:A\rtarr A\times A$ serves as a unit $1$-cell,  and it is helpful to  reinterpret composition in terms of the diagram \autoref{pbdiagtoo}.
\end{defn}

\begin{rem}\label{unitprob}
This bicategory is almost a $2$-category. The composition of spans is strictly associative, but if $|A|\geq 2$ then $\Delta_A:A\rtarr A\times A$ acts as a strict unit only on the right and so should be called a pseudo-unit $1$-cell.  The point is that with our chosen model for the pullback, the left map in the span composition
\[ \xymatrix @R=1em{
&& \Delta_B\circ E \ar[dl]_{p_1} \ar[dr]^{p_2} && \\
& B \ar@{=}[dl] \ar@{=}[dr] & & E \ar[dl]_f \ar[dr]^g & \\
B && B && A } 
\]
must be order-preserving. Therefore, if $f$ is not order-preserving, then $\Delta_B\circ E \neq E$. 
However, in view of the evident commutative diagram
\[ \xymatrix{
&  \Delta_B\circ E \ar[dl]_{p_1} \ar[dr]^{g\com p_2}  \ar[d]^{p_2} \\
B & E\ar[l]^{f} \ar[r]_{g} & A, \\} \]
the function $p_2$ specifies a reordering isomorphism of spans
\begin{equation}\label{laDe}
\xymatrix@1{ \Delta_B\circ E\ar[r]^-{\ell_{B,E}} & E \\} 
\end{equation} 
In \autoref{Whisker}, we show how to whisker the pseudo-unit $1$-cells to obtain an equivalent construction $\GEone$ that still has a strictly associative composition but now has strict two-sided unit $1$-cells.  The construction is closely analogous to the usual whiskering of a degenerate basepoint in a space to obtain a nondegenerate basepoint.  
\end{rem}

\begin{rem}\label{tenprod}  We are suppressing some categorical details that are irrelevant to our work.  The composition 
distributes over coproducts, and it should be defined on a ``tensor product'' rather than a cartesian product of 
permutative categories. Such a tensor product does in fact exist \cite{HyPow}, 
 in the sense that the 2-category of permutative categories has a pseudo-monoidal structure (\cite[Section~2.3]{HyPow}); however, we will not use this.
Rather, we will use that 
composition is a pairing that gives rise to a pairing defined on
the smash product of the spectra constructed from $G\sE(B,C)$ and $G\sE(A,B)$.  This passage from pairings of permutative
categories to pairings of spectra has a checkered history even nonequivariantly,\footnote{That starts from \cite{MayPair}, which is modernized, corrected, and generalized  in \cite{GMMOMain}, where pairings are subsumed as $2$-ary morphisms in multicategories.} and it is here that a mistake occurred in earlier versions of this paper.  As explained in \cite{GMMOMain}, categorical strictification and the
full development of multiplicative equivariant  infinite loop space theory  resolve the relevant issues.
\end{rem}

Before beginning work, we recall an old result that motivated this paper.
The category $[G\sE]$ of {isomorphism classes of} $G$-spans is obtained from the bicategory 
$G\sE$ of $G$-spans by identifying spans from $A$ to $B$ if there is an 
isomorphism between them.  Composition is again by pullbacks. 
We add spans from $A$ to $B$ by taking disjoint unions, and that
gives the morphism set $[G\sE](A,B)$ a structure of abelian monoid. 
We apply the Grothendieck construction to obtain an abelian group of 
morphisms $A\rtarr B$. This gives an additive category $\sA\!b[G\sE]$. 

\begin{defn}\label{GDcon}  Define $\GDAll$ to be the full subcategory of $G\sS$ whose objects are {fibrant replacements} of the $G$-spectra $\SI^{\infty}_G(A_+)$ in the stable model structure \cite{MM}, 
where $A$ runs over the finite $G$-sets, and let $\text{Ho}\GDAll\subset  \text{Ho}G\sS$ denote its homotopy category.
\end{defn}

\begin{thm}[{\cite[V.9.6]{LMS}}\footnote{{All $G$-spectra in \cite{LMS} are fibrant, but we are using orthogonal  $G$-spectra in this paper.  The homotopy categories are equivalent.}} ]\label{hostart} The categories $\text{Ho}\GDAll$ and $\sA\!b[G\sE]$ are 
isomorphic.
\end{thm}

\subsection{The precise statement of the main theorem}

Infinite loop space theory associates a spectrum $\bK \sA$ to a permutative category $\sA$.
There are several machines available and all are equivalent \cite{MayPerm2}. 
Since it is especially convenient for the equivariant generalization, we require $\bK$ to take 
values in
the category $\sS$ of
 orthogonal spectra \cite{MMSS}, {but symmetric spectra would also work.} 
Slightly modifying  the axiomatization of \cite{MayPerm2}, we 
require $\bK$ to take values in {positive}\footnote{{This means that $E_0 \rtarr \Omega E_1$ need not be an equivalence}} $\OM$-spectra 
and we require a natural map
$\et \colon B\sA \rtarr (\bK\sA)_0$
whose composition with
$(\bK\sA)_0\rtarr \OM (\bK\sA)_1$ gives a group completion.

Since $\sS$ is closed symmetric monoidal under the smash product, 
it makes sense to enrich categories in $\sS$. Our preferred version of 
spectral categories is categories enriched in $\sS$, abbreviated $\sS$-categories.  Model theoretically,
$\sS$ is a particularly nice enriching category since its unit $S$ is cofibrant in the stable model 
structure and $\sS$ satisfies the monoid axiom \cite[12.5]{MMSS}.  

When a spectral category $\sD$ is used as the domain category of a presheaf category, 
the objects and maps of the underlying category are unimportant. The important data are the 
morphism spectra $\sD(A,B)$, the unit maps $S\rtarr \sD(A,A)$, and the composition maps
\[\sD(B,C)\sma \sD(A,B)\rtarr \sD(A,C). \]  
The presheaves $\sD^{op}\rtarr \sS$ can be thought of as (right) $\sD$-modules.

Recall that an object $a$ in a permutative category $\sA$ 
determines a point of $B\sA$ hence, via $\et$, {a point}  of $(\bK\sA)_0$.  Therefore each $a\in \sA$ 
determines a map $S\rtarr \bK\sA$.
We will use this to specify unit maps for spectral categories.

\begin{defn}\label{KGE}  We define a spectral category $G\sA$. Its objects are the finite 
$G$-sets $A$, which may be viewed as the spectra $\bK G\sE(A)$.  Its morphism spectra 
are defined by
$G\sA(A,B) = \bK\GEone(A,B)$,
where $\GEone(A,B)$ is defined in \autoref{span1b}. Its unit maps
$S\rtarr G\sA(A,A)$ are induced by the {identity 1-cells}
in  $\GEone(A,A)$, and its composition
\[G\sA(B,C)\sma G\sA(A,B)\rtarr G\sA(A,C) \]  
is induced by composition in $\GEone$.
\end{defn}

As written, the definition makes little sense: to make the word ``induced'' meaningful 
requires a suitably behaved machine $\bK$, as we will spell out 
in \autoref{GMSum}. For the purpose of \autoref{KGE}, the machine of \cite{EM} would 
be sufficient, although it takes values in symmetric  rather than orthogonal spectra. 
However, the proof of our main theorem, given in \autoref{PROOF}, will use the 
equivariant machine of \cite{GMMOMain}, and we will therefore use the same 
machine to make sense of \autoref{KGE}. 
Once this is done, we will have the presheaf category $\mathbf{Pre}(G\sA,\sS)$ 
of $\sS$-functors $(G\sA)^{op}\rtarr \sS$  and $\sS$-natural transformations. As 
shown for example in \cite{GM0},
it is a cofibrantly generated model category enriched in $\sS$, or an $\sS$-model category for short.
As shown in \cite{MM}, the category $G\sS$  of (genuine) orthogonal $G$-spectra is also an $\sS$-model
category. Our main theorem can be restated as follows.

\begin{thm}[Main theorem]\label{MAINNew1}  There is a zigzag of enriched Quillen equivalences 
connecting the $\sS$-model categories $G\sS$ and $\mathbf{Pre}(G\sA,\sS)$.
\end{thm}

Therefore $G$-spectra can be thought of as constructed from the very elementary category $G\sE$
enriched in permutative categories, ordinary nonequivariant spectra,
and the black box of infinite loop space theory.  

We have chosen to take all finite $G$-sets $A$ as the objects of $G\sA$.  As we discuss in  \autoref{MAINOne}, \autoref{MAINNew1} holds just as well if we allow $A$ to instead range only over the orbits $G/H$ for subgroups $H\subset G$ (or even over one $H$ in each conjugacy class).  
As discussed in \autoref{bigboy}, this can be viewed as a consequence of the fact that the spectral enrichment forces additivity.
Intuitively,  a $G$-spectrum is then described by its fixed point spectra $X^H$, together with enriched restriction and transfer data.  A bit more precisely, let $\sO_G$ denote the category of orbits $G/H$ and $G$-maps between them.  For a $G$-spectrum $X$, passage to fixed point spectra specifies a contravariant functor  ${X^{(-)}\colon} \sO_G \rtarr \sS$.  The following reassuring result falls out of the proof of \autoref{MAINNew1}.  We shall be more precise about this in \autoref{reassure2}.

\begin{cor}\label{reassure}  The zigzag of equivalences induces a natural zigzag of equivalences 
between the fixed point orbit functor, 
$X\mapsto \{G/H\mapsto X^{H}\}$,
on $G$-spectra and the functor given by restricting presheaves $G\sA \rtarr \sS$ to {the (unenriched) orbit category}.
\end{cor}

Thus, if $X$ is a fibrant $G$-spectrum that corresponds to the presheaf $Y$, then $X^H$ is equivalent to $Y(G/H)$.

\begin{rem} For any $n$, the homotopy groups $\pi_n(X^H)$ define a Mackey functor, and so do 
the homotopy groups $\pi_n(Y(G/H))$.   The corollary implies an isomorphism between these Mackey
functors.
\end{rem}

We view \autoref{MAINNew1} as a $G$-spectrum analog of the standard equivalence between  $G$-spaces and space-valued presheaves on $\sO_G$
 (e.g. \cite[Chapter VI]{EHCT}).   As there, we do not in any sense view the theorem as giving a {\em replacement} for the category of $G$-spectra.  We regard $G$-spectra as natural objects of intrinsic interest, and their presheaf descriptions as an illuminating perspective.   We give some comparisons of functors to illustrate this in the brief \autoref{sec:Comparisons}.

\subsection{The $G$-bicategory $\sE_G$ of spans: intuitive definition}
\label{IntuitiveEG}
 
Everything we do depends on 
first working equivariantly and then passing to fixed points. 
We fix some generic notations. For a category $\sC$, let $G\sC$ be the category of 
$G$-objects in $\sC$ and $G$-maps between them. Let $\sC_G$ be the $G$-category of $G$-objects 
and nonequivariant maps, with $G$ acting on morphisms by conjugation. The two categories are related conceptually 
by $G\sC = (\sC_G)^G$. The objects, being $G$-objects, are already $G$-fixed; we apply the $G$-fixed point 
functor to hom sets. The reader may prefer to think of $\sC_G$ as a category enriched in $G$-categories, with enriched hom objects the $G$-categories $\sC_G(A,B)$ for $G$-objects $A$ and $B$.

We apply this framework to the category of finite $G$-sets.  We have already 
defined the $G$-fixed bicategory $G\sE$, and we shall give two definitions of 
$G$-bicategories $\sE_G$ with fixed point bicategories equivalent to $G\sE$.  The
first, given in this section, is more intuitive, but the second is 
more convenient for the proof of our main theorem. 

Let $U$ be a countable $G$-set that contains all orbit types $G/H$ infinitely many times.
Again let $A$, $B$, and $C$ denote finite $G$-sets, but now think of the $D$, $E$, and $F$ of 
(\ref{spans}) and (\ref{pbdiag}) as finite subsets of {the} $G$-set $U$; these subsets need {\em not} be $G$-subsets.
The action of $G$ on $U$ gives rise to an action of $G$ on the finite subsets of $U$: for
a finite subset $D$ of $U$ and $g\in G$, $gD$ is another finite subset of $U$.

\begin{defn} We define a $G$-{groupoid}
$\sE_G^U(A)$.
The objects of $\sE_G^U(A)$ are the 
nonequivariant maps $p\colon D\rtarr A$, where $A$ is a finite $G$-set and $D$ is
a finite subset of $U$.  The morphisms $f\colon p\rtarr q$, $q\colon E\rtarr A$, 
are the bijections $f\colon D\rtarr E$ such that $q\com f=p$.  The group $G$ acts 
on objects and morphisms by sending $D$ to $gD$ and sending a bijection 
$f\colon D\rtarr E$ over $A$ to the bijection $gf\colon gD\rtarr gE$ over $A$ given by $(gf)(gd) = g(f(d))$. 
\end{defn}

\begin{defn} We define a bicategory $\sE_G^U$ with objects the finite $G$-sets and 
with $G$-{groupoids}  
of morphisms between objects given by 
$\sE_G^U(A,B) = \sE_G^U(B\times A)$. Thinking of the objects of $\sE_G^U(A,B)$ as 
nonequivariant spans $B\ltarr D \rtarr A$, composition and units are 
defined as in  \autoref{span1a}.
\end{defn}

Observe that taking disjoint unions of finite sets over $A$ will not keep us in $U$
and is thus not well-defined. Therefore the $\sE_G^U(A)$ are not even symmetric monoidal 
(let alone permutative) $G$-categories in the naive sense of symmetric monoidal 
categories with $G$ acting compatibly on all data.  

\subsection{The $G$-bicategory $\sE_G$ of spans: working definition}

We shall work with a less intuitive definition of $\sE_G$, one that solves the problem of disjoint 
unions by avoiding any explicit use of them.  It uses an especially convenient $E_{\infty}$ 
operad of $G$-categories, denoted $\sP_G$. We recall it from  \cite{GM3},
where we define a genuine permutative $G$-category to be an algebra over $\sP_G$.  
More generally, in \cite{GMMO2} we define a genuine symmetric monoidal $G$-category to be a pseudoalgebra over $\sP_G$, but we will not  need that notion here.  Such pseudoalgebras provide input for an equivariant infinite loop space machine.

To define $\sP_G$, we apply our general point of view on equivariant categories to 
the category $\sC\!at$ of small categories.  Thus, for $G$-categories $\sA$ and $\sB$, 
let $\sC\!at_G(\sA,\sB)$ be the $G$-category of functors $\sA\rtarr \sB$ and natural 
transformations, with $G$ acting by conjugation,  and let $G\sC\!at(\sA,\sB)= C\!at_G(\sA,\sB)^G$ 
be the category of $G$-functors and $G$-natural transformations.

\begin{defn} Let $\tG$ 
be the groupoid\footnote{While $\tG$ is 
isomorphic as a $G$-category to the translation category of $G$, the action of $G$ on that category 
is defined differently, as is explained in \cite[Proposition 1.8]{GMM}.  Our $\tG$ is the chaotic category of
$G$, sometimes denoted $\tilde{G}$.}  with object set $G$ and a unique morphism, denoted $(h,k)$, from $k$ to $h$ for each pair of objects. Let $G$ act from the right on $\tG$ by $h\cdot g = hg$ on objects
and $(h,k)\cdot g = (hg,kg)$ on morphisms.  
Define $\sP(j) = \tSI_j$; this is the $j$th category of an $E_{\infty}$
operad of categories whose algebras are the permutative categories \cites{Dunn, MayPerm}.
Define $\sP_G(j)$ to be the $G$-category 
\[ \sP_G(j) = \sC\!at_G(\tG,\tSI_j) = \sC\! at_G(\tG,\sP(j) ). \]  
Here $G$ acts trivially on $\tSI_j$.  The left action of $G$ on $\sP_G(j)$ is 
induced by the right action of $G$ on $\tG$, and 
the right action of $\SI_j$ {is} induced 
by the right action of $\SI_j$ on $\tSI_j$.  The functor $\sC\!at_G(\tG,-)$ is product preserving
and the operad structure maps are induced from those of $\sP$. We interpret $\sP(0)$ and $\sP_G(0)$ to be trivial
categories; $\sP_G(1)$ is also trivial, with unique object denoted $\oid$.
\end{defn}  

\begin{defn}\label{OOO} Regard a finite $G$-set $A$ as a discrete $G$-{groupoid}  
(identity morphisms only).
Define the $G$-groupoid 
$\sE_G(A)$ by
\begin{equation}\label{yeah}
\sE_G(A) = \coprod_{n\geq 0} \sP_G(n)\times_{\SI_n} A^n = (\coprod_{n\geq 1} \sP_G(n)\times_{\SI_n} A^n)_+.
\end{equation}
We interpret the term with $n=0$ to be a trivial base category $\ast$, which explains the second equality, and we identify 
the term with $n=1$ with $A$.  
\end{defn}

In the language of \cite[Definition~4.5]{GM3}, $\sE_G(A)$ is the free genuine permutative $G$-groupoid generated by the $G$-set $A$;
its unit can be thought of as given by a disjoint trivial base category implicitly added to $A$. This is made precise by \autoref{yeah2}.

The following result is neither obvious nor difficult. It is proven in \cite{GM3}, where it is one
ingredient in a categorical proof of the tom Dieck splitting theorem.

\begin{thm}[{\cite[Theorem 5.9]{GM3}}]
\label{FixIso}  The $G$-fixed permutative groupoid 
$\sE_G(A)^G$ is naturally isomorphic 
to the permutative  groupoid 
 $G\sE(A)$ {of \autoref{PermCat}.}
\end{thm}

The starting point of the proof is the observation that a functor $\tG\rtarr \tSI_n$ is uniquely 
determined by its object function $G\rtarr \SI_n$. In particular, for a finite $G$-set {$B=\ub{n}^\be$} 
we may view the {group homomorphism $\be\colon G\rtarr \SI_n$ as an object of the category $\sP_G(n)$}.
With a little care, we see that a $G$-fixed object $(\be;a_1,\cdots,a_n)$ of $\sP_G(n)\times_{\SI_n} A^n$ can 
be interpreted as a $G$-map $B\rtarr A$ and that all finite $G$-sets over $A$ are of this form. 

\begin{rem}\label{OOO2}  Conceptually, \autoref{OOO} hides an important identification and extension of functoriality 
that will be used crucially in \autoref{subtlety}.   A priori, $\sE_G$ is a functor on {\em unbased} finite $G$-sets, but an alternative reformulation is
\begin{equation}\label{yeah2} 
\sE_G(A) = \bP_G(A_+), 
\end{equation}
where $\bP_G$ is the monad on the category 
of {\em based} $G$-categories, not just $G$-groupoids, whose algebras are the same as the $\sP_G$-algebras. Thus \autoref{yeah2} exhibits $\sE_G$ as a special case of the more general functor  $\bP_G$.  With this reinterpretation, $\sE_G(A)$ extends to a functor on all based finite $G$-sets 
and  all based $G$-maps, not just those of the form $f_+$.
\end{rem}

We need to be more precise about this identification and extended functoriality. 

\begin{defn}\label{ident}  
Let $\LA$ be the category of finite based sets ${\bn}$ and based injections.\footnote{{The category $\LA$ is isomorphic to the category of finite (unbased) sets and injections. We use based here both for historical reasons and because it fits better into the machinery of infinite loop space theory.}}  
For a finite based $G$-set $\cA$, regarded as a discrete based $G$-category, 
insertion of basepoints makes the powers $\cA^n$ into a covariant functor $\cA^{\bullet}$ from $\LA$ to based $G$-categories.
Then $\bP_G(\cA)$ is the categorical tensor product 
\[ \bP_G(\cA) = \sP_G(\bullet)\otimes_{\LA} \cA^{\bullet}. \]
Since any based injection $\si\in \LA(\bm,\bn)$ can be written uniquely as the composition of a permutation of $\bm$ followed by an order-preserving injection, the contravariant functoriality of $\sP_G(\bullet)$ on based injections is given by combining the right $\SI_m$-action on $\sP_G(m)$ with the contravariant functoriality with regards to ordered injections described in \cite[2.3]{MayGeo}.
Thus
\begin{equation}\label{yeah3yack}  
\bP_G(\cA) =  \Big(\coprod_{n\geq 0} \sP_G(n)\times \cA^n \Big)/(\sim), 
\end{equation}
where 
\[(\si^*\mu; a)\sim (\mu;\si_* a) \quad \text{for} \quad \mu\in \sP_G(n),\quad \si\in \LA(\bm,\bn),\quad\text{and}\quad a\in \cA^m.\]
As in \cite[2.3]{MayGeo}, we can first pass to orbits using the permutations in $\LA$ and then use the equivalence relation induced by the proper injections to rewrite this as
\begin{equation}\label{yeah3}  
\bP_G(\cA) =  \Big(\coprod_{n\geq 0} \sP_G(n)\times_{\SI_n} \cA^n \Big)/(\sim), 
\end{equation}
thus highlighting the comparison with \autoref{yeah}.
\end{defn}

\begin{defn}\label{subtlety}
 For a based $G$-map $f\colon A_+\rtarr B_+$, define a functor 
\[ f_!\colon \sE_G(A)\rtarr \sE_G(B) \]
using the identification \autoref{yeah2} and the functoriality of $\bP_G$ on based maps.\footnote{With the intuitive version of $\sE_G$ described in \autoref{IntuitiveEG},
 $f_!\colon \sE_G(A)\rtarr \sE_G(B)$ is then just the pushforward functor obtained by composing maps over $A$ with $f$.}
In the case that $f^{-1}(\ast) = \ast$, so that $f$ is in the image of the disjoint basepoint functor $X\mapsto X_+ $, the functor $f_!$ is given by the disjoint union over $n$ of the functors 
\[ \sP_G(n) \times_{\SI_n} A^n\xrightarrow{\id\times_{\SI_n} f^n} \sP_G(n) \times_{\SI_n} B^n.\]
If $i\colon A\rtarr B$ is an injection of unbased 
finite $G$-sets, define an associated retraction $r\colon B_+\rtarr A_+$ of based finite $G$-sets by setting $ri(a) = a$ 
and $r(b) = \ast$ if $b\notin \im(i)$. Then define
\[ i^* = r_!\colon \sE_G(B) \rtarr \sE_G(A).\footnote{With the intuitive version of $\sE_G$ described in \autoref{IntuitiveEG},
$i^*\colon \sE_G(B)\rtarr \sE_G(A)$ is just the functor obtained by using $i$ to pull back maps over $B$ to maps over $A$.} \]
By \autoref{transfer} below, we may think of $i^*$ as the dual of $i$.
\end{defn}

The following definition gives the $G$-category analogue of \autoref{span1a}.  It specifies a $G$-category (almost) 
``enriched in permutative $G$-categories''.

\begin{defn}\label{compcat}
We define a $G$-bicategory\footnote{{As in \autoref{unitprob}, the bicategory $\sE_G$ is almost a $2$-category. It is just missing  strict units, as we shall explain shortly.}} 
$\sE_G$ with a permutative $G$-groupoid hom object for each pair of objects as follows.
The $0$-cells of $\sE_G$ are the finite $G$-sets $A$, which may be thought of as the $G$-categories $\sE_G(A)$.  The 
permutative $G$-groupoid 
 $\sE_G(A,B)$ of $1$-cells and $2$-cells $A\rtarr B$ is $\sE_G(B\times A)$, as defined in \autoref{OOO}.  
The composition 
\[ \circ\colon \sE_G(B,C)\times \sE_G(A,B)\rtarr \sE_G(A,C) \]
is given by the {diagram \eqref{compsquare}.} 
Its first map $\om$ is a pairing of free $\sP_G$-algebras that will be made precise
in \autoref{omega}. Its second and third maps implement composition from the point of view of \autoref{pbdiagtoo}. They are specializations of the contravariant functoriality of $\sE_G$
on injections and its covariant functoriality on surjections, as is made precise in \autoref{subtlety}.
\begin{equation}\label{compsquare}
 \xymatrix @C=5em {
\sE_G(C\times B)\sma \sE_G(B\times A)\ar[d]_{\om} \ar@{-->}[r]^-{\circ} & 
\sE_G(C\times A).\\
 \sE_G(C\times B\times B\times A) \ar[r]^-{(\id\times \DE\times\id)^*} & 
\sE_G(C\times B\times A) \ar[u]_{\pi_!}
} \end{equation}
This composition is strictly associative, as we indicate in  
\autoref{leftright}.  With {$A = \ub{n}^{\al},$} $\sE_G(A,A)$ has a pseudo-unit $1$ cell  
\begin{equation}\label{unitA}
\DE_A = (\al;\DE_A) \in \sE_G(A\times A) = \sP_G(n)\times_{\SI_n} (A\times A)^n
\end{equation}
where 
\[\DE_A = \big((1,1),\cdots,(n,n)\big) \in (A\times A)^n.\]
It is a strict right unit, but it is not a strict left unit (see \autoref{leftright}). 
\end{defn}

To rectify to obtain a strict unit, we need whiskered $G$-categories $\EGone$ analogous to the whiskered categories $\GEone$, 
and we define them in  \autoref{Whisker}.  They are defined in such a way that \autoref{FixIso} has the following 
corollary by direct comparison of definitions.

\begin{cor}\label{FixIso2} The $G$-fixed category $(\EGone)^G$ enriched in permutative 
categories is isomorphic to the category $\GEone$ enriched in permutative categories.
\end{cor}

In \autoref{omega} we will give an ad hoc definition of the pairing $\om$  that is required in \autoref{compcat}.  
We  place $\om$ 
in a general multicategorical context in \cite[Definition~3.20]{GMMOMain}.  We first comment on its domain; compare \autoref{tenprod}.

\begin{rem} We can define the smash product of based $G$-categories in the same way as the 
smash product of based $G$-spaces (see \cite[Lemma~4.20]{EM}).  We are most interested in examples of the form $\sA_+$ and
$\sB_+$ for unbased $G$-categories $\sA$ and $\sB$, and then $\sA_+\sma \sB_+$ can be identified with $(\sA\times \sB)_+$.  
Therefore
\small{
\begin{equation}\label{EGAEGB_reorder}
\xymatrix @R=1em{ 
 \sE_G(A) \sma \sE_G(B) =
\Big(\coprod_{m\geq 1} \sP_G(m)\times_{\SI_m} A^m\Big)_+\sma \Big(\coprod_{n\geq 1} \sP_G(n)\times_{\SI_n} B^n \Big)_+
\ar[d]^\iso \\
 \Big(\coprod_{m\geq 1,n\geq 1} \sP_G(m)\times_{\SI_m} A^m \times \sP_G(n) \times_{\SI_n} B^n \Big)_+  \ar[d]^{\iso}\\
  \Big(\coprod_{m\geq 1,n\geq 1} \sP_G(m)\times\sP_G(n)\times_{\SI_m\times \SI_n} A^m\times B^n \Big)_+ \\} 
 \end{equation}
 }
Note that this smash product does not have a $\sP_G$-algebra structure.
\end{rem}

\begin{defn}\label{omega} 
The homomorphism $\otimes:\SI_m\times\SI_n\rtarr \SI_{mn}$ defined using lexicographic ordering in \autoref{Biperm} is the object function of a functor 
\[\tSI_m\times\tSI_n\rtarr \tSI_{mn}.\]
Applying the functor $\sC\!at_G(\tG,-)$, we obtain pairings  
\[\otimes\colon \sP_G(m) \times \sP_G(n)\rtarr \sP_G(mn);\]
on objects of $\tG$, $(\mu\otimes \nu)(g) = \mu(g)\otimes \nu(g)$.
{For $G$-sets $A$ and $B$, we have the injection 
\[\boxtimes:A^m\times B^n\rtarr (A\times B)^{mn}\]
that sends $(a_1,\cdots,a_m)\times (b_1,\cdots, b_n)$ to the set of pairs $(a_i,b_j)$, ordered lexicographically.}
Combining, there result functors
\[  \om_{m,n}\colon (\sP_G(m)\times_{\SI_m} A^m)\times (\sP_G(n)\times_{\SI_n} B^n)\rtarr \sP_G(mn)\times_{\SI_{mn}} (A\times B)^{mn}, \]
\[ \om_{m,n}\big( (\mu,a),(\nu,b)\big) = (\mu\otimes \nu,a\boxtimes b).\]
Using the description \autoref{EGAEGB_reorder}, the functors $\om_{m,n}$
 specify pairings of $G$-categories
\[ \om \colon \sE_G(A) \sma \sE_G(B) \rtarr \sE_G(A\times B). \]
\end{defn}

While $\sE_G(A) \sma \sE_G(B)$ is not a $\sP_G$-algebra, we show in \cite[Proposition~3.25]{GMMOMain} that $\om$ is an example of a bilinear, or 2-ary, morphism in the multicategory of $\sP_G$-algebras. The machine of \cite{GMMOMain} then produces from this bilinear map a pairing of $G$-spectra, as we will discuss in \autoref{GMSum} below.

\begin{rem}\label{leftright}
The associativity of the composition $\circ$ defined in \autoref{compcat} is an easy diagram chase, starting from
the associativity of the pairing on $\sP_G$.  We illustrate how \autoref{subtlety} works by considering composites
with the pseudo-unit objects $\DE_A$.  Let $E$ be a $1$-cell in $\sE_G(A,B)$ and choose an object
\[(\mu;(b_1,a_1),\cdots,(b_m,a_m)) \ \ \text{of} \ \   \sP_G(m)\times (B\times A)^m\]
in the $\SI_m$-orbit $E$.

We first prove that $E\com \DE_A = E$.  Suppose that {$A = \ub{n}^\al$.} Then the object 
\[ \big{(} \mu\otimes\al; ((b_i,a_i,j,j))\big{)} \ \ \text{of} \ \  \sP_G(mn)\times (B\times A\times A\times A)^{mn} \]  
is in the $\SI_{mn}$-orbit $\om(E,\DE_A)$. The ordering of the four-tuples is lexicographic on  $i$ and $j$.  The four-tuple 
$(b_i,a_i,j,j)$ is in the image of $\id\times\DE\times \id$ if and only if $a_i = j$.  The retraction
corresponding to this 
injection maps 
such a $(b_i,a_i,j,j)$ to $(b_i,a_i,j)=(b_i,a_i,a_i)$ and
all other $(b_i,a_i,j,j)$ to the basepoint. Applying $\pi_!$  we arrive at 
\[\si_*((b_1,a_1),\cdots, (b_m,a_m)) \in (B\times A)^{mn},\]
where  $\si\colon \bf{m} \rtarr \bf{mn}$ is the ordered injection that sends $i$ to $\la^{-1}_{m,n}(i,a_i)$.   Therefore
\[ E\circ \DE_A = (\mu\otimes \al; \si_*((b_1,a_1),\cdots, (b_m,a_m))) = (\si^*(\mu\otimes \al); (b_1,a_1),\cdots, (b_m,a_m)). \]
Since $\si^*$ reverses the lexicographic ordering used to define $\mu\otimes \al$, we have the relation   $\si^*(\mu\otimes \al) = \mu$.

Now let {$B = \ub{p}^\be$} and consider $\DE_B \com E$.  Then the object
\[ \big{(} \be\otimes \mu; (k,k,b_i,a_i))\big{)} \ \ \text{of}  \ \ \sP_G(pm)\times (B\times B\times B\times A)^{pm}\]  
is in the $\SI_{pm}$-orbit $\om(\DE_B,E)$. The ordering of the four-tuples is lexicographic on  $k$ and $i$.  The four-tuple 
$(k,k,b_i,a_i)$ is in the image of $\id\times\DE\times \id$ if and only if $k=b_i$.  The retraction corresponding to this 
injection maps all other $(k,k,b_i,a_i)$ to the basepoint. Applying $\pi_!$  we arrive at
\[\ta_*((b_1,a_1),\cdots, (b_m,a_m)) \in (B\times A)^{pm},\]
where $\ta\colon \bf{m}\rtarr \bf{pm}$ is the injection that sends $i$ to $\la^{-1}_{p,m}(b_i,i)$.  
We have
\[ \DE_B\circ E =(\be\otimes \mu, \ta_*((b_1,a_1),\cdots, (b_m,a_m))= (\ta^*(\be\otimes \mu); (b_1,a_1),\cdots, (b_m,a_m)), \]
but the injection $\ta$ is not ordered, and
$\ta^*(\be\otimes \mu)$ is not equal to $\mu$.  
We define 
\begin{equation}
\ell_{B,E}:\DE_B\circ E \rtarr E
\label{laDeG}\end{equation} 
to be the 2-cell induced by the (unique) morphism $\tau^*(\be\otimes\mu) \rtarr \mu$ in $\sP_G(m)$.
The structure $\sE_G$ is only a bicategory, while $\EGone$ defined in \autoref{Whisker} is a strict 2-category. 
The inclusion $\sE_G \rtarr \EGone$ is a pseudofunctor with unit constraint given by $\ze$ of \autoref{def:whisker}.
In \cite{GMMOMain}, 
the category of $\sP_G$-algebras is  
given the structure of a multicategory.
The composition functors in both $\sE_G$ and $\EGone$ are examples of bilinear maps in the
multicategorical sense.
\end{rem}

\subsection{The categorical duality maps}\label{duality}

Since various specializations are central to our work, we briefly recall how duality 
works categorically, following \cite[III\S1]{LMS} for example.  We then define maps
of $\sP_G$-algebras that will lead in \autoref{selfish} to the proof that the objects of $G\sA$ are self-dual. 

Let $\sV$ be a closed symmetric monoidal category with product $\sma$, unit $S$, 
and hom objects $F(X,Y)$; write $DX = F(X,S)$.  A pair of objects $(X,Y)$ in $\sV$ 
is a dual pair if there are maps $\et\colon S\rtarr X\sma Y$ and $\epz\colon Y\sma X\rtarr S$
such that the composites
\[  \xymatrix@1{ 
X \iso S \sma X \ar[rr]^-{\et\sma \id} && X\sma Y\sma X \ar[rr]^-{\id\sma \epz} 
&& X\sma S \iso X } \]
\[  \xymatrix@1{ Y\iso Y\sma S \ar[rr]^-{\id \sma \et} && Y\sma X\sma Y \ar[rr]^-{\epz\sma \id} 
&& S\sma Y \iso Y } \]
are identity maps.  For any such pair, the adjoint $\tilde{\epz}\colon Y\rtarr DX$ of $\epz$ 
is an isomorphism.  When $(X,Y)$ and $(X',Y')$ are dual pairs, the dual of a map $f\colon X\rtarr X'$ 
is the composite
\begin{equation}\label{dualmap}
 \xymatrix@1{Y'\iso Y'\sma S_G \ar[r]^-{\id\sma \et} & Y'\sma X\sma Y \ar[r]^-{\id\sma f\sma \id} 
& Y'\sma X'\sma Y\ar[r]^-{\epz'\sma\id} & S_G\sma Y\iso Y. }
\end{equation}

For any pair of objects $X$ and $Z$, we  have a natural map
\begin{equation}\label{zeta} 
\ze\colon Z\sma DX = Z\sma F(X,S) \rtarr F(X,Z)
\end{equation}
in $\sV$, namely the adjoint of
\[ \id\sma \epz\colon Z\sma DX \sma X  \rtarr  Z\sma S  \iso Z, \]
where $\epz$ is the evident evaluation map. The map $\ze$ is an isomorphism
when either $X$ or $Z$ is dualizable \cite[III.1.3]{LMS}.  When $X$ is self-dual and 
$Z$ is arbitrary, we have the composite isomorphism
\begin{equation}\label{delta}
\de = \ze \com (\id\sma \tilde{\epz}) \colon Z\sma X \rtarr Z\sma DX \rtarr F(X,Z).
\end{equation}
This map in various categories will play an important role in our work.

In \autoref{epsilon} and \autoref{eta}, we will define
two maps of $\sP_G$-algebras that are central to duality and therefore to 
everything we do.  Let $S^0 = \{\ast,1\}$, where $\ast$ is the basepoint and ${1}$
is not.  We think of $S^0$ as $1_+$, where $1$ is the one-point $G$-set.  In line with
this convention, we also think of $1$ as a trivial category with object $1$.
Remember that $\sE_G(A) = \bP_G(A_+)$ is the free $\sP_G$-algebra generated by 
$A_+$, where we view finite $G$-sets as categories with only identity morphisms.

\begin{defn}\label{epsilon} For a finite $G$-set {$A=\ub{n}^\al$,} define based $G$-maps
\[ \epz\colon (A\times A)_+ \rtarr S^0, \ \ r\colon (A\times A)_+ \rtarr A_+ \ \  \text{and}\ \ \pi\colon A_+\rtarr S^0\]
by $r(a,b) = \ast$ if $a\neq b$ and $r(a,a) = a$, $\pi(a) = 1$, and 
$\epz = \pi\circ r$, so that $\epz(a,b) = \ast$ if $a\neq b$ and $\epz(a,a) = 1$.
Note that $r\com \DE_+ = \id_{A_+}$.
We agree to again write $\epz$ for the induced map of 
$\sP_G$-algebras 
\[\epz = \sE_G \epz\colon \sE_G(A\times A) \rtarr \sE_G(1). \]
\end{defn}

\begin{defn}\label{eta} For a finite $G$-set {$A=\ub{n}^\al$,} regard the object $\DE_A\in \sE_G(A\times A)$ 
as the map of $G$-categories $i_A\colon 1 \rtarr \sE_G(A\times A)$ that sends the object $1$ of the trivial category 
to the object $\DE_A$. By freeness, there results a map of $\sP_G$-algebras 
{\[\et\colon \sE_G(1) \rtarr \sE_G(A\times A).\]}
Explicitly,\footnote{{This uses that $\ga(\mu;\al^n) = \mu\otimes\al$, where $\ga\colon \sP_G(m)\times \sP_G(n)^m\rtarr \sP_G(mn)$, as explained in \cite[\S3.1]{GMMOMain}.}} 
 $\eta$ is  
 the disjoint union over $m$ of the maps 
\[\sP_G(m)\times_{\SI_m} 1^m\rtarr \sP_G(mn)\times_{\SI_{mn}} (A\times A)^{mn}\]
given by
\[ \et(\mu,1^m) = \big(\mu\otimes \al; (\DE_A)^m\big). \]

\end{defn}
The following categorical observation will lead to our proof in \autoref{selfish} that the $G$-spectra $\SI^{\infty}_G(A_+)$
are self-dual.  Since care of basepoints is crucial, we use the alternative notation $\bP_G(A_+)$. Remember
that $(A\times A)_+$ can be identified with $A_+\sma A_+$.  We identify $1_+\sma A_+$ and
$A_+\sma 1_+$ with $A_+$ at the bottom center of our diagrams. 

\begin{prop}\label{NewDual} In the diagrams below, square (1) commutes up to isomorphism, and the other three squares commute on the nose.
\[ \xymatrix{
\bP_G(A_+\sma A_+) \sma \bP_G(A_+) \ar[r]^-{\om}  \drtwocell<\omit>{<0> \ \ (1) } & \bP_G(A_+\sma A_+\sma A_+) \ar[d]^{\bP_G(\id\sma\epz)}
& \bP_G(A_+)\sma \bP_G(A_+\sma A_+) \ar[d]^{\id\sma \epz} \ar[l]_-{\om}\\
\bP_G(1_+) \sma \bP_G(A_+) \ar[r]_-{\om} \ar[u]^{\et\sma\id} &
 \bP_G(A_+) & \bP_G (A_+)\sma \bP_G(1_+) \ar[l]^-{\om} } \]

\[ \xymatrix{
\bP_G(A_+) \sma \bP_G(A_+\sma A_+) \ar[r]^-{\om} \ar@{}[dr]|{(2)} 
& \bP_G(A_+\sma A_+\sma A_+) \ar[d]^{\bP_G(\epz\sma\id)}
& \bP_G(A_+\sma A_+)\sma \bP_G(A_+) \ar[d]^{\epz\sma \id} \ar[l]_-{\om}\\
\bP_G(A_+) \sma \bP_G(1_+) \ar[r]_-{\om} \ar[u]^{\id\sma\et} &
\bP_G(A_+) & \bP_G (1_+)\sma \bP_G(A_+) \ar[l]^-{\om} } \]
\end{prop}

\begin{proof} In the right vertical arrows, $\epz$ means $\bP_G(\epz)$.  
Both right squares are naturality diagrams, so it remains to consider the squares on the left. 
The difference between squares (1) and (2) is closely analogous to the difference between
left and right composition with $\DE_A$, as explained in \autoref{leftright}.  Let $A=\ub{n}^\al$ and 
consider objects  $(\mu,1^m)$ of $\sP(m)\times 1^m$ and $(\nu,a)$ of $\sP(q)\times A^q$. We consider
square (2) first, paying close attention to the order in which variables appear.

By  Definitions \ref{omega} and \ref{eta},
\[ \om\big((\nu,a),(\mu,1^m)\big) = (\nu\otimes \mu, a\boxtimes 1^m) \ \ \text{in} \ \  \sP(qm)\times A^{qm}\]
and
\[ \om\com (\id\sma\et)\big((\nu,a),(\mu,1^m)\big) =\big(\nu\otimes \mu\otimes \al; a\boxtimes (\DE_A)^m) \ \ \text{in}\ \ \sP_G(qmn)\times_{\SI_{qmn}} (A^3)^{qmn}. \]
Identifying {$\ub{qm}$ with $\ub{q}\times \ub{m}$} lexicographically, the $(k,i)$th coordinate of $a\boxtimes 1^m$ is $a_k$.  Identifying 
{$\ub{qmn}$ with $\ub q\times \ub m\times \ub n$}
lexicographically, the $(k,j,i)$th coordinate of \linebreak
$a\boxtimes (\DE_A)^m$ is $(a_k,i,i)$.  
By \autoref{epsilon}, $\epz\sma \id$ sends this coordinate to the basepoint unless $a_k = i$, when it sends it to $i$.  Noticing the agreement
of lexicographic orderings, we see as in \autoref{leftright} that the injection $\si\colon \bf{qm} \rtarr \bf{qmn}$ such that
\[ \si_*(a\boxtimes 1^m) = (\epz\sma\id)_*(a\boxtimes (\DE_A)^m)\]
is ordered and satisfies $\si^*(\nu\otimes \mu\otimes \al) = \nu\otimes \mu$.  

Now consider square (1). By  \autoref{omega} and \autoref{eta},
\[ \om\big( (\mu,1^m),(\nu,a)\big) = \big(\mu\otimes \nu, 1^m\boxtimes a\big) \ \ \text{in} \ \ \sP(mq)\times_{\SI_{mq}} A^{mq} \]
and
\[ \om\com (\et\sma\id)\big( (\mu,1^m),(\nu,a)\big) =\big(\ga(\mu;\al^n)\otimes \nu; (\DE_A)^m\boxtimes a\big) \ \ \text{in} \ \  \sP_G(mnq)\times_{\SI_{mnq}} (A^3)^{mnq}.\]
Identifying {$\ub{mq}$ with $\ub{m}\times \ub{q}$} lexicographically, the $(i,k)$th coordinate of $1^m\boxtimes a$ is $a_k$.
Identifying {$\ub{mnq}$ with $\ub{m}\times \ub{n}\times \ub{q}$} lexicographically, the $(i,j,k)$th coordinate of $(\DE_A)^m\boxtimes a$  is $(j,j,a_k)$.  
By \autoref{epsilon}, $\id\sma \epz$ sends this coordinate to the basepoint unless $j=a_k$, when it sends it to $j$.  Here the injection $\ta\colon \bf{mq} \rtarr \bf{mnq}$
such that 
\[ \ta (1^m\boxtimes a) = (\id\sma\epz)_*((\DE_A)^m\boxtimes a )\]
is not ordered, 
and $\ta^*(\mu\otimes \al \otimes \nu)$ is not equal to $\mu\otimes \nu$ in $\sP_G(mq)$. 
As in \autoref{leftright}, there is a unique $2$-cell, necessarily an isomorphism, 
\[ \vartheta\colon (\mu\otimes \nu) \Longrightarrow \ta^*(\mu\otimes \al\otimes \nu) \]
in $\sP_G(mq)$.  As the input varies, the $2$-cells
\[ (\vartheta,\id)\colon \big(\mu\otimes \nu; 1^m\boxtimes a) \Longrightarrow \big(\ta^*(\mu\otimes \al\otimes \nu),1^m\boxtimes a\big) \]
specify the $2$-natural isomorphism  in the square (1).
\end{proof}

\section{The proof of the main theorem}
\label{sec:MainThm}

\subsection{The equivariant approach to \autoref{MAINNew1}}
\label{EquivApproachMAINNew1}

As we explain in  \cite{GMMOMain}, following \cite{GM3}, equivariant infinite loop space theory associates an 
orthogonal $G$-spectrum $\bK_G\sC_G$ to a genuine permutative (or more generally genuine symmetric monoidal) $G$-category $\sC_G$.  
The map $B\sC_G = (\bK_G \sC_G)_0 \rtarr \OM (\bK_G \sC_G)_1$ is an equivariant group completion.
\footnote{{The papers 
from around 1990, such as \cites{CostWan, Shim} are not adequate, in part because genuine permutative $G$-categories were not
explicitly defined and the group completion property was not worked out rigorously, but more substantially because a symmetric monoidal category 
of $G$-spectra had not yet been discovered.  A key feature of the version of the Segal machine \cite{GMMO1} used in our proofs is that it is given by a 
symmetric monoidal functor, a claim that would not  have made sense in 1990.}}

\begin{notn}
We denote by $G\sS$ the (closed symmetric monoidal) category of orthogonal $G$-spectra, indexed on a complete universe, and maps of such. 
A category enriched over $G\sS$ will be referred to as a $G\sS$-category.

The category $G\sS$ has two futher relevant enrichments. Using the closed structure yields a self-enrichment, which we write as $\sS_G$. Thus, for $G$-spectra $X$ and $Y$, the $G$-spectrum $\sS_G(X,Y)$ is the mapping $G$-spectrum $F_G(X,Y)$. Applying fixed points to the mapping $G$-spectra gives a $\sS$-enriched category, which we again write as $G\sS$. This parallels the discussion at the start of \autoref{IntuitiveEG}.
\end{notn}

Applying the functor $\bK_G$ to $\sE_G$ (\autoref{compcat}), we obtain the following 
equivariant analogue of \autoref{KGE}.  

\begin{defn}\label{GKGE}  We define a $G$-spectral category, 
or $G\sS$-category, $\sA_G$.  Its objects are the finite 
$G$-sets $A$, which may be viewed as the $G$-spectra $\bK_G\sE_G(A)$.  Its morphism $G$-spectra 
$\sA_G(A,B)$ are the $\bK_G\EGone(B\times A)$. Its unit $G$-maps $S_G\rtarr \sA_G(A,A)$ 
are induced by the points $I_A\in \GEone(A,A)$ (see \autoref{Whisker}) and its composition $G$-maps
\[\sA_G(B,C)\sma \sA_G(A,B)\rtarr \sA_G(A,C) \]  
are induced by composition in $\EGone$.
\end{defn}

Again, as written, the definition makes little sense: to make the word ``induced'' meaningful 
requires properties of the equivariant infinite loop space machine $\bK_G$ that we will spell out
in \autoref{GMSum}.  This depends on having a functor
that takes pairings (alias bilinear maps)
of free $\sP_G$-algebras 
to pairings of $G$-spectra.    

The equivariant and non-equivariant infinite loop space functors are related by the following result.

\begin{thm}[\cite{GM3}]\label{start}  There is a natural equivalence of spectra
\[ \io\colon \bK(G\sC) \rtarr (\bK_G\sC_G)^G \]
for permutative $G$-categories $\sC_G$ with $G$-fixed permutative categories $G\sC$.
\end{thm}

In view of \autoref{FixIso2}, there results an equivalence of $\sS$-categories
\[\xymatrix@1{ G\sA \ar[r]^-{\simeq} &  (\sA_G)^G.} \]

The proof of \autoref{MAINNew1} goes as follows. 
 We now write $\GDAll$ for the spectral version of the category introduced in \autoref{GDcon}.
 We start with the following {\autoref{MAINOneCor}, which is} a specialization
of \cite[Lemma~1.35]{GM0}; it is discussed in \autoref{PresheafModels}.  
The essential point is that the collection $\{\SI^\infty_G A_+\}$ is a set of generators for $\text{Ho}G\sS$.

\begin{thm}\label{MAINOneCor}  There is an $\sS$-enriched Quillen adjunction
\[\xymatrix@1{\mathbf{Pre}(\GDAll,\sS) \ar@<.4ex>[r]^-\bT & G\sS \ar@<.4ex>[l]^-{\bU}, }\]
and it is a Quillen equivalence.
\end{thm}

\begin{rem}\label{AllOrb}   Instead of using $\GDAll$, we can use its full subcategory $\GDOrb$ 
obtained by restricting the $A$ to be orbits $G/H$, and then the result generalizes to compact Lie 
groups $G$; see \autoref{MAINOne}.  We define
$\GDOrb$ as we defined $\GDAll$ in \autoref{GDcon}, again using fibrant replacements.    
 Then $\GDAll$ and $\GDOrb$ are the $G$-fixed $\sS$-categories obtained from full
 $G\sS$-subcategories $\sD_{All}$ and $\sD_{Orb}$ of $\sS_G$.
\end{rem}

We will prove the following result in \autoref{PROOF}.

\begin{thm}[Equivariant version of the main theorem]\label{MAINNew2}  There is a zigzag of 
weak equivalences connecting the $G\sS$-categories $\sA_G$ and $\sD_{All}$.
\end{thm} 

A weak equivalence between $G\sS$-categories with the same object sets
is just an $G\sS$-enriched functor that induces weak equivalences on morphism $G$-spectra.\footnote{A
more general definition is given in \cite[Definition~2.3]{GM0}.}  On passage to $G$-fixed categories, this 
equivariant zigzag induces a zigzag of weak $\sS$-equivalences connecting the $\sS$-categories 
$G\sA$ and $\GDAll$. In turn, by \cite[Proposition~2.4]{GM0}, this zigzag induces a zigzag of Quillen equivalences 
between $\mathbf{Pre}(G\sA,\sS)$ and $\mathbf{Pre}(\GDAll,\sS)$.  Since $\mathbf{Pre}(\GDAll,\sS)$ is 
Quillen equivalent to $G\sS$, it follows that \autoref{MAINNew2} implies \autoref{MAINNew1}.

\begin{rem} 
For a $G$-spectrum X, the functor $\bU(X)$ (of \autoref{MAINOneCor})
sends an orbit $G/H$ to $F_G(\SI^{\infty}_G G/H_+,X)^G\iso X^H$. Keeping that
fact in mind shows why \autoref{reassure} follows from the proof of \autoref{MAINNew1}.
\end{rem}

To understand $G\sS$ as an $\sS$-category, we must first understand $\sS_G$ as a $G\sS$-category.
That is, to understand the $G$-fixed spectra $F_G(X,Y)^G$, we must first understand the function $G$-spectra 
$F_G(X,Y)$.  Using infinite loop space theory to model function spectra implicitly raises a conceptual issue: 
there is no known infinite loop space machine that knows about function spectra. That is, given input data $X$ 
and $Y$ (permutative $G$-categories, $E_{\infty}$-$G$-spaces, $\GA$-$G$-spaces, etc) for an infinite 
loop space machine $\bK_G$, we do not know what input data will have as output the function $G$-spectra 
$F_G(\bK_GX,\bK_GY)$. The problem does not even make sense as just stated because the output $G$-spectra 
$\bK_GX$ are always connective, whereas $F_G(\bK_GX,\bK_GY)$ is generally not.  The most that one 
could hope for in general is to detect the connective cover of $F(\bK_GX,\bK_GY)$. In our case, the 
relevant function $G$-spectra are connective since the suspension $G$-spectra $\SI^{\infty}_G(A_+)$ are
self-dual, as we shall reprove in \autoref{selfish}.

\subsection{Results from equivariant infinite loop space theory}\label{GMSum}
The proof of \autoref{MAINNew2} is the heart of this paper, and of course   
it depends on equivariant infinite loop space theory and in particular on the relationship 
between the $G$-spectra $\sA_G(A)=\bK_G\sE_G(A)$ and the suspension $G$-spectra $\SI^{\infty}_G(A_+)$.  
We collect the results that we need from \cite{GMMOMain} in this section. We warn the skeptical reader 
that the results of this paper depend fundamentally on  Theorems \ref{MultiFun} and \ref{BPQ0}.
However, the proofs of those results require work far afield from the applications in this paper.

In fact, \autoref{MAINNew2} is an application of a categorical version of the 
equivariant Barratt-Priddy-Quillen (BPQ) theorem for the identification of suspension 
$G$-spectra.\footnote{For $A=\ast$, Carlsson \cite[p.6]{Carl} mentions a space level version 
of the BPQ theorem. Shimakawa \cite[p. 242]{Shim} states and gives a sketch 
proof of a $G$-spectrum level version.}  We state the theorem
in full generality before restricting attention to finite $G$-sets.  We shall find use for the
full generality in \autoref{secsusp}. 

Recall from \autoref{OOO2} that $\sE_G(A)$ can be identified with the category $\bP_G(A_+)$, where $\bP_G$ is
the free $\sP_G$-category functor on based $G$-categories.  The functor $\bP_G$ applies equally well to based 
topological $G$-categories.\footnote{We understand a topological $G$-category to mean an internal category in 
the category of $G$-spaces.} We view a based $G$-space $X$ as a topological 
$G$-category that is discrete in the categorical sense: its morphism and object $G$-spaces are both $X$, 
and its source, target, identity, and composition maps are all its identity map.  Thus we have the 
topological $\sP_G$-category $\bP_G(X)$. The geometric realization of its nerve is the free
$E_{\infty}$-$G$-space generated by $X$.

Henceforward, we use the term stable equivalence, rather than weak 
equivalence, for the weak equivalences in our model categories of spectra and $G$-spectra.
In previous work, we established \cite[Theorem~6.2]{GM3} an equivariant version of the Barratt-Priddy-Quillen theorem, 
giving a natural equivalence between $\Sigma^\infty_G X_+$ and $\bK_G \bP_G(X)$.
However, in order to produce our spectral category $\sA_G$, 
we require a more structured version of that result.

First, it is essential that we have a machine with good multiplicative properties.  The following result, which is proven 
in \cite{GMMOMain}, gives far more than we need.  As explained in \cite[\S3]{GMMOMain}, we have a multicategory  
$\Mult(\sP_G)$ of (strict) $\sP_G$-algebras and pseudomorphisms between them; it is a submulticategory of a 
multicategory $\Mult(\sP_G\text{-}\mathbf{PsAlg})$ of $\sP_G$-pseudoalgebras. 
The multilinear maps of $\Mult(\sP_G)$
 require $\sP_G$-pseudomaps despite the restriction to strict $\sP_G$-algebras as objects.   We also have the multicategory 
$\Mult(G\sS)$ associated to the symmetric monoidal category of orthogonal $G$-spectra under the smash product.

\begin{thm}\cite{GMMOMain}\label{MultiFun}
$\bK_G$ extends to a multifunctor
\[ \bK_G: \Mult(\sP_G) \rtarr \Mult(G\sS).\]
\end{thm}

\begin{rem}\label{homotopy}
At one place in the duality proof of \autoref{selfish} below, we use from \cite[Proposition~9.24]{GMMOMain} 
that $\bK_G$ converts 
$2$-cells, such as $\vartheta$ 
in the proof of \autoref{NewDual}, to homotopies between maps of $G$-spectra. 
\end{rem}

\begin{rem}
\label{KGposfibrant}
 In the proof of \autoref{MAINNew2}, we will  use the fact that $\bK_G$ takes values in positive $\Omega$-$G$-spectra \cite{GMMOMain}.
\end{rem}

\begin{cor}\label{AGspectral}
The construction $\sA_G$ given in \autoref{GKGE} defines a $G\sS$-category.
\end{cor}
\begin{proof}
It is shown in \cite[\S3.5]{GMMOMain} that the pairing $\om$ of \autoref{omega} is a bilinear morphism in $\Mult(\sP_G)$.
Moreover, the functors $(\id\times\Delta\times\id)^*$ and $\pi_{!}$ of \autoref{compsquare} are maps of $\sP_G$-algebras. It 
follows that the composition $\sE_G(B,C)\times\sE_G(A,B)\xrightarrow{\circ} \sE_G(A,C)$ is also bilinear. 
This remains true after applying the whiskering construction of \autoref{Whisker}.
Therefore the multifunctor $\bK_G$ produces a map of $G$-spectra $\sA_G(B,C) \sma \sA_G(A,B) \rtarr \sA_G(A,C)$ as desired.
The fact that the composition in $\EGone$ is strictly associative and unital ensures that the same is true in $\sA_G$.
\end{proof}

\autoref{MultiFun} yields another important consequence. 
Observe that the pairing $\om$ of \autoref{omega} generalizes from $G$-sets $A$ and $B$ to 
$G$-spaces $X$ and $Y$, giving a natural pairing
\[ \om\colon \bP_G(X_+)\sma\bP_G(Y_+)\rtarr \bP_G(X_+\sma Y_+). \] 
Then \autoref{MultiFun} produces a map of $G$-spectra
\[ \sma\colon \bK_G\bP_G(X_+) \sma \bK_G\bP_G(Y_+) \rtarr \bK_G\bP_G(X_+\sma Y_+). \]
This  makes the assignment $X\mapsto \bK_G \bP_G(X_+)$ into a lax monoidal functor from (unbased) $G$-spaces to orthogonal $G$-spectra.

With this multiplicative machine in hand, it now makes sense to ask for a Barratt-Priddy-Quillen comparison that 
is also compatible with the multiplicative structure. That is another main result of \cite{GMMOMain}. 
Recall that the assignment $X \mapsto \Sigma^\infty_G X_+$ is a strong monoidal functor from (unbased) $G$-spaces to orthogonal $G$-spectra.

\begin{thm}[\cite{GMMOMain}]
\label{BPQ0} 
There is a monoidal natural transformation
\[  \al\colon \SI^{\infty}_G(X_+)\rtarr \bK_G\bP_G(X_+)  \]
of functors from (unbased) $G$-spaces to orthogonal $G$-spectra, which restricts to a natural stable equivalence on the subcategory of $G$-CW complexes.
\end{thm}

For the remainder of this section, we restrict our attention to the case when $X$ is a finite $G$-set $A$, although we will return to the generality of $G$-spaces $X$ in \autoref{secsusp}. We therefore use the identification \eqref{yeah2} to rewrite $\bP_G(A_+)$ as $\sE_G(A)$. 

That the transformation of \autoref{BPQ0} is monoidal 
means that we have a commutative diagram
\begin{equation}\label{accept1} \xymatrix{
\SI^{\infty}_G(A_+) \sma \SI^{\infty}_G(B_+)\ar[d]_{\sma}^-{\iso}  \ar[r]^-{\al\sma \al} 
&  \bK_G\sE_G(A) \sma \bK_G\sE_G(B) \ar[d]^{\sma}\\
\SI^{\infty}_G(A \times B)_+ \ar[r]_-{\al} & \bK_G\sE_G(A\times B). } 
\end{equation}

We restate the naturality of $\al$ with respect to $G$-maps $f\colon A\rtarr B$ 
in the diagram
\begin{equation}\label{accept2} 
\xymatrix{
\SI^{\infty}_G(A_+) \ar[d]_{\SI^{\infty}_G f_+}  \ar[r]^-{\al} 
&  \bK_G\sE_G(A) \ar[d]^{\bK_G f_!}\\
\SI^{\infty}_G(B_+) \ar[r]_-{\al} & \bK_G\sE_G(B). }
\end{equation}

If $i\colon A\rtarr B$ is an injection with retraction $r\colon B_+\rtarr A_+$, 
we have the induced map of $G$-spectra
\[ \bK_Gi^* = \bK_G r_!\colon \bK_G\sE_G(B)\rtarr \bK_G\sE_G(A), \] 
and (\ref{accept2}) specializes to
\begin{equation}\label{accept3} 
\xymatrix{
\SI^{\infty}_G(B_+) \ar[d]_{\SI^{\infty}_Gr}  \ar[r]^-{\al} 
&  \bK_G\sE_G(B) \ar[d]^{\bK_G i^*}\\
\SI^{\infty}_G(A_+) \ar[r]_-{\al} & \bK_G\sE_G(A) }
\end{equation}  
By \autoref{transfer} below, we may identify $\bK_Gi^*$ as the dual of $\bK_Gi$
and thus $\SI^{\infty}_G r$ as the dual of $\SI^{\infty}_G i_+$.

We combine these diagrams to construct those that we need to prove \autoref{MAINNew2}.
Let $A$, $B$, and $C$ be finite $G$-sets and recall \autoref{compcat}.   

\begin{prop}\label{keykey}  The following diagram of $G$-spectra commutes.
\begin{equation}\label{keydiagram} \xymatrix{
\SI^{\infty}_G(C\times B)_+ \sma \SI^{\infty}_G(B\times A)_+ 
\ar[d]_{\sma}^{\iso} \ar[r]^-{\al\sma\al} 
& \bK_G\sE_G(C\times B) \sma \bK_G\sE_G(B\times A) \ar[d]^{\sma} \\
\SI^{\infty}(C\times B\times B\times A)_+ \ar[r]^-{\al} \ar[d]_{\SI^{\infty}_G r}
& \bK_G\sE_G(C\times B\times B\times A) \ar[d]^{\bK_G(\id\times \DE\times \id)^*}\\
\SI^{\infty}(C\times B\times A)_+ \ar[r]^-{\al} \ar[d]_{\SI^{\infty}\pi}
& \bK_G\sE_G(C\times B\times A) \ar[d]^{\bK_G \pi_!} \\ 
\SI^{\infty}_G(C \times A)_+ \ar[r]_-{\al}  & \bK_G\sE_G(C\times A) }
\end{equation}
Here $r$ is the retraction which sends the complement of the image of
$\id\times \DE\times\id$ to the basepoint.
\end{prop}

The diagram (\ref{keydiagram}) relates the composition pairing of the $G\sS$-category
$\sA_G$ to remarkably simple and explicit maps between suspension $G$-spectra.  In fact, 
recalling \autoref{epsilon} and again writing $\epz=\SI^{\infty}_G\epz$, we see that the 
left vertical composite in (\ref{keydiagram}) can be identified with 
$\id\sma\epz\sma\id$.  We have proven the following result,
where we abuse notation by writing $\alpha$ for the composite
\[ \SI^\infty_G(B\times A)_+ \rtarr \bK_G\sE_G(B\times A) \rtarr \bK_G\EGone(B\times A).\]

\begin{thm}\label{curious} The following diagram of $G$-spectra commutes in $HoG\sS$.
\[ \xymatrix{
\SI^{\infty}_G(C\times B)_+ \sma \SI^{\infty}_G(B\times A)_+ \ar[d]_{\iso} \ar[r]^-{\al\sma\al} 
& \sA_G(B,C) \sma \sA_G(A,B) \ar[ddd]^{\circ} \\
\SI^{\infty}_G (C_+)\sma \SI^{\infty}_G(B\times B)_+ \sma \SI^{\infty}_G(A_+)   \ar[d]_{\id\sma \epz\sma\id} & \\
\SI^{\infty}_G (C_+) \sma S_G \sma \SI^{\infty}_G(A_+)\ar[d]_{\iso} & \\
\SI^{\infty}_G(C\times A)_+ \ar[r]_-{\al}  & \sA_G(A,C) } \]
\end{thm}

\subsection{The self-duality of $\SI^{\infty}_G(A_+)$}\label{selfish}

Let $A$ be a finite $G$-set and write $\bA = \SI^{\infty}_G(A_+)$ for brevity of notation.  
As recalled in \autoref{duality}, in order to show that $\bA$ is self-dual in $HoG\sS$,
we must define maps
$\et\colon  S_G \rtarr \bA\sma \bA$  and $\epz\colon \bA\sma \bA \rtarr S_G$
in the stable homotopy category $HoG\sS$ such that the composites
\begin{equation}\label{dualid}
\xymatrix@1{ \bA \ar[r]^-{\et\sma \id} & \bA\sma \bA \sma \bA \ar[r]^-{\id\sma\epz} & \bA } \ \ \tand \ \
\xymatrix@1{ \bA \ar[r]^-{\id\sma \et} & \bA\sma \bA \sma \bA \ar[r]^-{\epz\sma\id} & \bA }
\end{equation}
are the identity map in $HoG\sS$.  Using the stable equivalence $\al$ and the definitions of $\et$ and $\epz$
from  \autoref{epsilon} and \autoref{eta}, we let $\et$ and $\epz$ be the composites
\[\xymatrix@1{ 
S_G \ar[r]^-{\al} & \bK_G\sE_G(1) \ar[r]^-{\bK_G\et} & \bK_G\sE_G(A\times A) \ar[r]^-{\al^{-1}} 
& \SI^{\infty}_G (A\times A)_+ \iso \bA\sma\bA } \]
and
\[\xymatrix@1{ 
\bA\sma\bA \iso \SI^{\infty}_G (A\times A)_+ \ar[r]^-{\al} & \bK_G\sE_G(A\times A) \ar[r]^-{\bK_G\epz} 
& \bK_G\sE_G(1) \ar[r]^-{\al^{-1}} & S_G. } \]
The following commutative diagram proves that the first composite in (\ref{dualid}) is the identity map in $HoG\sS$;
the second is dealt with similarly.
We abbreviate notation by setting $\sA_GA = \bK_G\sE_G(A)$.  Remember that $\sE_G(A) = \bP_G(A_+)$. The center two
squares are derived by use of the diagrams from \autoref{NewDual}.

\[ \scalebox{0.8}{\xymatrix @C1em{
\sA_G(A^2)\sma \bA \ar[dr]^{\id\sma \al} & & 
({\bA^2})\sma \bA \iso {\bA^3} \iso \bA\sma ({\bA^2}) 
\ar[ll]_-{\al\sma \id} \ar[rr]^-{\id\sma\al} \ar[d]^{\al} & &
\bA\sma \sA_G(A^2) \ar[dl]_{\al\sma\id} \ar[dd]^{\id\sma \epz}\\
& \sA_G(A^2)\sma\sA_G A \ar[r]^{\sma} & \sA_G(A^3) \ar@<1ex>[dd]^{\id\times\epz} 
& \sA_G A\sma \sA_G(A^2) \ar[l]_{\sma} \ar[dd]^{\id\sma \epz} & \\
\sA_G 1 \sma \bA \ar[ur]_{\et\sma\al} \ar[uu]^{\et\sma\id} & & & & \bA \sma \sA_G 1 \ar[dl]_{\al\sma\id} \\
& \sA_G1 \sma \sA_G A \ar[uu]_{\et\sma\id} \ar[r]^{\sma} & \sA_G A 
& \sA_G A\sma \sA_G1 \ar[l]_{\sma} & \\
S_G\sma \bA \ar[ur]_{\al\sma\al} \ar[rr]_{\iso} \ar[uu]^{\al\sma\id} 
&& \bA \ar[u]_{\al} 
&& \bA\sma S_G \ar[ul]^{\al\sma\al} \ar[ll]^{\iso} \ar[uu]_{\id\sma\al} }} \]

Given \autoref{BPQ0}, it is trivial that the outer parts of the diagram commute. The right central diagram 
is just a naturality diagram, as in \autoref{NewDual}.   The left central diagram commutes up to homotopy
by that result and \autoref{homotopy}.  

Specializing general observations about duality recalled in \autoref{duality}, we have the following corollary.
This homotopical input is the crux of the proof of \autoref{MAINNew2}. 

\begin{cor}\label{DELTA} 
For finite $G$-sets $A$ and $B$, the canonical map  
\[\delta = \ze\com (\id\sma \tilde{\epz}) \colon \bB\sma \bA\rtarr \bB\sma D\bA \rtarr F_G(\bA,\bB)\]
of (\ref{delta}) is a stable equivalence.   
\end{cor}

We insert a mild digression concerning the identification of some of our maps.

\begin{rem}\label{transfer} For an injection $i\colon A\rtarr B$ of finite $G$-sets, (\ref{dualmap}) and the 
precise constructions of $\et$ and $\epz$ starting from  \autoref{epsilon} and \autoref{eta} imply that the 
dual of $i$ is the map $\bB\rtarr \bA$ induced by the evident retraction $r\colon B_+\rtarr A_+$.
A $G$-map $\pi\colon G/H\rtarr G/K$ is a bundle, and the dual of $\SI^{\infty}\pi_+$ is the associated 
transfer map (see e.g. \cite[IV.pp 182 and 192]{LMS}).  It can be identified explicitly by a similar
(but not especially illuminating) inspection of definitions. \end{rem}

\subsection{The proof that $\sA_G$ is equivalent to $\sD_{\text{All}}$}\label{PROOF}

We will have to chase large diagrams, and we again abbreviate notations by writing
\[ \bA = \SI^{\infty}_G(A_+), \ \ \  \bB = \SI^{\infty}_G(B_+), 
\ \ \ \text{and} \ \ \ \bC = \SI^{\infty}_G(C_+)\]
for finite $G$-sets $A$, $B$, and $C$.  We also abbreviate notation by writing
\[ \sA_G(A)=\sA_G(\ast,A).\] 
This is the $G$-spectrum $\sA_G(A)=\bK_G\sE_G(A)$, which is equivalent to $\bA$ by \autoref{BPQ0}.  
Remember that we are free to choose any bifibrant equivalents of the $G$-spectra $\bA$
as the objects of $\sD_{\text{All}}$.  

\begin{proof}[Proof of \autoref{MAINNew2}] We use model categorical arguments,
and we work with the stable model structure on $G\sS$. We use \cite[\S2.4]{GM0} to obtain
a model structure on the category $G\sS\bO$-$\sC\! at$ of $G\sS$-categories with the
same object set $\bO$ as $G\sE$.  {We emphasize that this is a model structure on a 
category of categories.}  Maps are weak equivalences or fibrations if they
induce weak equivalences or fibrations on hom objects in $G\sS$.  Here the nature of 
the objects is irrelevant; we are concerned with $G\sS$-categories with one object for 
each finite $G$-set $A$. 

Let $\la\colon Q{\sA_G}\rtarr \sA_G$ be a cofibrant approximation of $\sA_G$.  
By \cite[Theorem~2.16]{GM0}, since $S_G$ is cofibrant in the stable model structure 
each morphism $G$-spectrum $Q\sA_G(A,B)$ is cofibrant in $G\sS$. The maps 
$\la\colon Q\sA_G(A,B) \rtarr \sA_G(A,B)$ are stable acyclic fibrations.
Digressively, since the $\sA_G(A,B)$ are fibrant in the positive stable model 
structure (see \autoref{KGposfibrant}), that is also true of the $Q\sA_G(A,B)$; we will use this fact later, 
in \autoref{secsusp}.

Let $\rh\colon Q\sA_G\rtarr RQ\sA_G$ be a fibrant approximation of $Q\sA_G$.  
The morphism $G$-spectra $RQ\sA_G(A,B)$ are then bifibrant in the stable
model structure.  Therefore $RQ\sA_G(A)$ is bifibrant for each $A$, and it is 
stably equivalent to $\bA$. 
We take the $RQ\sA_G(A)$ as the bifibrant approximations of the 
$\bA$ that we use to define the full $G\sS$-subcategory $\sD_{\text{All}}$ of $G\sS$.  

We now have a zig-zag
\[ \xymatrix{ \sA_G & Q\sA_G \ar[l]_\la^\sim \ar[r]^\rho_\sim & RQ\sA_G}\]
of stable equivalences of $G\sS$-categories. It remains to find a stable equivalence $RQ\sA_G \rtarr \sD_{\text{All}}$.
To abbreviate notation, let us write
$  RQ\sA_G (\ast,A) = RQ\sA_G A, $
and let
\[ \ga\colon RQ\sA_G(A,B) \rtarr \sD_{\text{All}}(A,B) = F_G(RQ\sA_GA,RQ\sA_GB) \]
be the adjoint of the composition map
\[ \circ\colon RQ\sA_G(A,B)\sma RQ\sA_GA\rtarr RQ\sA_GB. \]
By \cite[Construction~5.6]{GM0}, this defines a $G\sS$-functor
\[ \ga\colon RQ\sA_G\rtarr \sD_{\text{All}}.\]
It suffices to prove that each of the maps $\ga$ is a stable equivalence. 

We define $\mathcal{Q}_G$ to be
 the full $G\sS$-subcategory of $\sS_G$ with objects the $Q\sA_G(A)$.
It will play a role in our proof that $\ga$ is a stable equivalence.
To abbreviate notation, we agree to write $Q\sA_G(\ast,A) = Q\sA_G A$.
For finite $G$-sets $A$ and $B$, let
 \[ \be\colon Q\sA_G(A,B) \rtarr \mathcal{Q}_G(A,B) = F_G(Q\sA_GA,Q\sA_GB) \]
be the adjoint of the composition map
\[ \circ\colon Q\sA_G(A,B)\sma Q\sA_GA\rtarr Q\sA_GB. \]
This defines a $G\sS$-functor
\[ \be\colon Q\sA_G\rtarr \mathcal{Q}_G.\]
For each finite $G$-set $A$, $\bA$ is cofibrant and $\la\colon Q\sA_G A\rtarr \sA_G A$ 
is an acyclic fibration in the stable model structure on $G\sS$. Therefore there is a map 
$\mu\colon \bA\rtarr Q\sA_G A$ such that the diagram
\[ \xymatrix{
& Q\sA_G A \ar[d]^{\la} \\
\bA \ar[ur]^-{\mu} \ar[r]_-{\al} & \sA_G A } \]
commutes. Since $\al$ and $\la$ are stable equivalences, so is $\mu$. 
In the same way, we get a stable equivalence $\mu: \bB \sma \bA \rtarr Q\sA_G(A,B)$.

For the remainder of the proof, we work in the homotopy category $HoG\sS$. In particular, the distinction between $\bK_G\sE_G$ and $\bK_G\EGone$ vanishes. 
We claim that the following diagram of $G$-spectra commutes in $HoG\sS$. Indeed, modulo inversion of maps 
which are stable equivalences, it commutes on the nose. As before, we identify
$\bB\sma\bA = \SI^{\infty}_GB_+\sma \SI^{\infty}_GA_+$ with $\SI^{\infty}_G(B\times A)_+$.

\[ \xymatrix{
RQ\sA_G(A,B) \ar[r]^-{\ga} & F_G(RQ\sA_G A, RQ\sA_G B)
\ar[r]^-{F_{G}(\rh,\id)}_-{\htp} &
F_G(Q\sA_G A,RQ\sA_G B) \ar[dd]^-{F_{G}(\mu,\id)}_{\htp}  \\
&&\\
Q\sA_G(A,B) \ar[uu]^{\rh}_-{\htp} \ar[r]^-{\be}  
& F_G(Q\sA_G A,Q\sA_G B) \ar[uur]_-{F_G(\id,\rh)} \ar[ddr]^-{F_G(\mu,\id)} 
&  F_G(\bA,RQ\sA_G B) \\
&&\\
\bB\sma \bA \ar[uu]^{\mu}_{\htp} \ar[r]_-{\delta}^-{\htp} 
& F_G(\bA,\bB) \ar[r]_-{F_G(\id,\mu)}^-{\htp} 
& F_G(\bA, Q\sA_G B) \ar[uu]_-{F_G(\id,\rh)}^-{\htp}  } \]

The map $\de$ is the stable equivalence of \autoref{DELTA}. The maps $\mu$ and $\rh$ are also
stable equivalences.  The maps $F_G(\rh,\id)$ and $F_G(\mu,\id)$ that are 
labeled $\htp$ are stable equivalences by \cite[Lemma~1.22]{GM0} since $\rh$ and $\mu$
are maps between cofibrant objects and $RQ\sA_G B$ is fibrant.  The maps $F_G(\id,\mu)$ 
and $F_G(\id,\rh)$ that are labeled $\htp$ are stable equivalences by 
\cite[III.3.9]{MM}, which shows that the functor $F_G(\bA,-)$ preserves stable
equivalences. Provided that the diagram commutes, it follows that $\ga$ is a stable 
equivalence since all of the other outer arrows of the diagram are stable equivalences.

The top pentagon commutes since $\rho$ is a map of $G\sS$-categories, and both composites on the right give $F_G(\mu,\rho)$. 
It therefore remains to consider the lower pentagon.
To prove that the diagram commutes in $HoG\sS$, we consider its adjoint, 
which is displayed as the outer rectangle of the diagram below.
Here we have inserted the map $\com\colon\sA_G(A,B)\sma \sA_G A \rtarr \sA_G B$
and  arrows $\lambda$
into its source and target for purposes of proof.

\[ \xymatrix{
 Q\sA_G(A,B)\sma Q\sA_G A \ar[rrr]^-{\circ} \ar[dr]^-{\la\sma\la} 
 & & & Q\sA_G B  \ar[dl]_{\la}  \\
& \sA_G(A,B)\sma \sA_G A \ar[r]^-{\circ} & \sA_G B &  \\
\bB\sma \bA\sma \bA \ar[rrr]_{\id\sma \SI^{\infty}_G\epz}  
\ar[ur]^-{\al\sma\al} \ar[uu]^{\mu\sma\mu} & & &  \bB \ar[ul]_-{\al}  \ar [uu]_{\mu}
}\]

\noindent
Since $\la$ is a map of $G\sS$-categories, it is apparent
that all parts of the diagram commute except for the
bottom trapezoid. Taking $(A,B,C) = (\ast,A,B)$ in \autoref{curious}, we see that 
the trapezoid commutes.  Since the wrong way map $\la$ is a stable
equivalence and can be inverted upon passage to the homotopy category, this
diagram and its adjoint commute there.
\end{proof}

\subsection{The identification of suspension $G$-spectra}\label{secsusp}

We expand the adjoint $\sS$-equivalences in \autoref{MAINNew1} more explicitly as follows, using \cite[Proposition~2.4]{GM0}.
\begin{equation}\label{long}
\xymatrix{
G\sS \ar@<-.4ex>[r]_-\bU &
\mathbf{Pre}(\GDAll,\sS) \ar@<-.4ex>[r]_-{\gamma^*} \ar@<-.4ex>[l]_-{\bT}  &
\mathbf{Pre}((RQ\sA_G)^G,\sS)  \ar@<-.4ex>[d]_-{\rh^*} \ar@<-.4ex>[l]_-{\gamma_!} \\ 
\mathbf{Pre}(G\sA,\sS) \ar@<+.4ex>[r]^-{\iota_!} & 
\mathbf{Pre}((\sA_G)^G,\sS)\ar@<-.4ex>[r]_-{\lambda^*} \ar@<+.4ex>[l]^-{\iota^*}   &
\mathbf{Pre}((Q\sA_G)^G,\sS)  \ar@<-.4ex>[l]_-{\lambda_!} \ar@<-.4ex>[u]_-{\rho_!} \\ } 
\end{equation}
The map $\iota:G\sA\rtarr (\sA_G)^G$ is the equivalence of \autoref{start}. Before passage to $G$-fixed points,
the proof in \autoref{PROOF} gives stable equivalences of  $G\sS$-categories
\[ \rho\colon Q\sA_G\rtarr RQ\sA_G, \ \ \ga:RQ\sA_G\rtarr \sD_{\text{All}},  \ \text{and}\   \la\colon Q\sA_G\rtarr \sA_G. \]
These maps give stable equivalences of $\sS$-categories after passage to fixed points. 
Seeing this uses that the hom $G$-spectra in $RQ\sA_G$ and $\sD_{\text{All}}$ are fibrant, while those in $Q\sA_G$ and $\sA_G$ are positive fibrant, as discussed in the proof of \autoref{MAINNew2}.

For a finite $G$-set $B$, $\SI^{\infty}_GB_+$ corresponds under this zigzag to the presheaf $\mathbf{B}$ that sends 
$A$ to $G\sA(A,B)$.  This is almost a tautology since, for $E\in G\sS$, $\bU(E)$ is the presheaf represented by $E$,
while $G\sE(-,B)$ is the functor represented by $B$. In the proof of \autoref{MAINNew2}, we chose the 
bifibrant approximation of $\SI^{\infty}_GB_+$ in $\sD_{\text{All}}$ to be $RQ\sA_G(B)$. With $B$ fixed, that proof shows that 
$\ga$ gives an equivalence of presheaves
\[RQ\sA_G(-,B) \rtarr \gamma^*\bU RQ\sA_G(B)\] 
(before passage to $G$-fixed points). The functors $\rho^*$ and $\lambda_!$ and the isomorphism $\io^*$ preserve 
representable functors, and therefore $\io^*\lambda_!\rho^*RQ\sA_G(-,B)\simeq \bK_G\sE_G(-,B)$.

This observation can be generalized from finite based $G$-sets $B_+$ to arbitrary based $G$-spaces $X$.  To see
this, we mix general based $G$-spaces $X$ with finite based $G$-sets $A_+$ to obtain a functorial construction of a
presheaf $\PR(X)$.

\begin{defn} 
For a based $G$-space $X$,
define a presheaf $\PR(X)\colon (\sA_G)^{op} \rtarr \sS_G$ by letting
\[ \PR(X)(A) = \bK_G \bP_G (X\sma A_+). \]
The contravariant functoriality map
\[ \PR(X)\colon \sA_G(A,B) \rtarr F_G(\PR(X)(B),\PR(X)(A)) \]
is the composite of the retraction $\sA_G(A,B) = \bK_G\EGone(A,B) \rtarr \bK_G(\sE_G(B\times A))$ 
(see \autoref{span1b}) with 
the adjoint of the right vertical composite in the commutative diagram 
\begin{equation}\label{curiouser} 
\xymatrix{
\SI^{\infty}_G(X\sma B_+) \sma \SI^{\infty}_G(B_+\sma A_+) 
\ar[d]_{\sma}^{\iso} \ar[r]^-{\al\sma\al} 
& \bK_G\bP_G(X\sma B_+) \sma \bK_G\bP_G(B_+\sma A_+) \ar[d]^{\sma} \\
\SI^{\infty}(X\sma B_+\sma  B_+\sma  A_+) \ar[r]^-{\al} \ar[d]_{\SI^{\infty}_G r}
& \bK_G\bP_G(X\sma B_+\sma  B_+\sma  A_+) \ar[d]^{\bK_G\bP_G(r)}\\
\SI^{\infty}(X \sma  B_+\sma  A_+) \ar[r]^-{\al} \ar[d]_{\SI^{\infty}\pi}
& \bK_G\bP_G(X \sma  B_+\sma  A_+) \ar[d]^{\bK_G\bP_G \pi} \\ 
\SI^{\infty}_G(X \sma A_+) \ar[r]_-{\al}  & \bK_G\bP_G(X\sma A_+). } 
\end{equation}
Here $r$ is the 
retraction of based $G$-sets associated to the diagonal inclusion
and $\pi$ is the projection.
The diagram commutes by the same concatenation of commutative diagrams as in 
\autoref{keykey}.
Note that there is no need to whisker the $G$-categories $\bP_G(X\sma A_+)$ in order to get a strict functor. The spans in $\bP_G(X\sma A_+)$ are only composed on the right with spans in $\sA_G$ in this construction, and the $\Delta_B$ were already strict units on the right. Therefore use of the retraction does not destroy functoriality.
\end{defn} 

\begin{thm} Let $X$ be a based $G$-space.  Under our zigzag of equivalences, $\SI^{\infty}_G X$ 
corresponds naturally to the presheaf $(\PR(X))^G$ that sends $A$ to $\bK \big(\bP_G(X\sma A_+)^G\big)$.
\end{thm}
\begin{proof}  
Note that $\bK_G\bP_G(X\sma -_+)$ is no longer a representable presheaf.  We again work with $G$-spectra 
and obtain the conclusion after passage to $G$-fixed spectra.
According to \autoref{BPQ0}, we may replace $\SI_G^\infty X$ by the positive fibrant $G$-spectrum 
$\bK_G\bP_G(X)$, which we abbreviate to $\sA_G(X)$ by a slight abuse of notation. After this
replacement, the presheaf $\bU(\SI_G^\infty X)$ may be computed as
\[ \bU(\SI_G^\infty X)(A) = F_G(RQ\sA_G(A),\sA_G(X) ).\]
Therefore, following the chain of (\ref{long}), we may compute $\rh^*\ga^*\bU(\SI_G^\infty X)$ as
\[ \rho^*\ga^*\bU(\SI_G^\infty X) \simeq F_G( Q\sA_G(-),\sA_G(X)).\]
Replacing $(B,A)$ by $(A,1)$ in (\ref{curiouser}) and recalling that $1_+ = S^0$, the right column gives the 
second map in the composite
\begin{equation}\label{sups} 
\xymatrix@1{
\PR(X)(A)\sma Q\sA_G(A) \ar[r]^-{\id\sma \la} &\PR(X)(A)\sma \sA_G(A)\ar[r]^-{\circ} & \PR(X)(1).}
\end{equation}
Its target is the $G$-spectrum $\sA_G(X)$, and its adjoint gives a map of presheaves 
\begin{equation}\label{map} \la^* \PR(X)\rtarr F_G( Q\sA_G(-),\sA_G(X))
\end{equation}
with domain $Q\sA_G$.  It remains to show that this map is an equivalence. To compute the
adjoint (\ref{map}), observe that (\ref{sups}) is the top horizontal composite in the diagram
\[\xymatrix{
\PR(X)(A)\sma Q\sA_G(A) \ar[r]^{\id\sma \la} & \PR(X)(A) \sma \sA_G(A) \ar[r]^-{\circ} & \PR(X)(1) \\
\SI^\infty_G (X\sma A_+)\sma Q\sA_G(A) \ar[u]^{\al\sma \id} & \PR(X)(A)\sma \SI^{\infty}_GA_+  \ar[u]_{\id\sma \al} \\
\SI^\infty_G(X\sma A_+) \sma \SI^{\infty}_GA_+ \ar[u]^{\id\sma \mu} \ar[ur]_{\al\sma \id} 
\ar[r]_-{\iso} & \SI^\infty_GX \sma \SI^{\infty}_G(A_+\sma A_+) \ar[r]_-{\id\sma\epz}  
 &\SI^\infty_G X. \ar[uu]_\al 
} \]
The left pentagon commutes since $\la\com \mu = \al$ and the right pentagon is a special case of (\ref{curiouser}).
Therefore the map (\ref{map}) is the top horizontal composite in the diagram
\[ \xymatrix{
\PR(X)(A) \ar[r] & F_G(\sA_G(A),\sA_G(X) ) \ar[r]^-{F_G(\la,\id)} & F_G(Q\sA_G(A),\sA_G(X)) \ar[d]^{F_G(\mu,\id)} \\
\SI^\infty_G (X\sma A_+) \ar[u]^\al \ar[r]_-\delta & F_G(\SI^\infty_GA_+,\SI^\infty_G X) 
\ar[r]_-{F_G(\id,\al)}& F_G(\SI^\infty_G A_+,\sA_G(X)).
}\]
The map $\al$ is a stable equivalence by \autoref{BPQ0}. The map $\delta$ is the stable equivalence of (\ref{delta}). 
The map $F_G(\id,\al)$ is a stable equivalence by \cite[III.3.9]{MM}. Finally, the map $F_G(\mu,\id)$ is a stable 
equivalence by \cite[Lemma~1.22]{GM0}.
\end{proof}

\section{Some comparisons of functors}\label{sec:Comparisons}

\subsection{Change of groups {and fixed point} functors}

We discuss several constructions on $G$-spectra from the point of view of \autoref{MAINNew1}. 
Categorical fixed points are already built into the setup: for any subgroup $H\subset G$, the functor of $H$-fixed points is given by evaluating presheaves at the orbit $G/H$. We will return to this in \autoref{FixedPoints}.

\begin{con}[Restriction to subgroups]\label{Restriction} Let $H\subset G$ be a subgroup. Then induction of $G$-sets provides a strong monoidal (in other words, coproduct-preserving) bifunctor $G\times_H(-)\colon H\sE \rtarr G\sE$. 
Using our models for $H\sE$ and $G\sE$, we must declare a preferred ordering for an induced $G$-set $G\times_H A$, given an ordering of the $H$-set $A$. For this, we choose an ordering of  $G/H$ as well as a set of coset representatives for $H$ in $G$. The choice of coset representatives gives a bijection of sets $G\times_H A \iso G/H \times A$, and we use the lexicographic ordering of $G/H \times A$ to order the induced $G$-set $G\times_H A$.

This extends to a (strict) 2-functor $G\times_H-\colon H\Eone \rtarr \GEone$ if, recalling that the 1-cell $I_A \in H\Eone(A,A)$ is the identity  of $A$  as in \autoref{def:whisker}, we then define $G\times_H I_A = I_{G\times_H A}$ 
for all $H$-sets $A$. For finite $H$-sets $A$ and $B$, there is a unique $G$-equivariant isomorphism $G\times_H(A\amalg B) \iso (G\times_HA) \amalg (G\times_HB)$, though it is not order-preserving in general. It follows that the induction functor gives rise to a spectral functor $\bK(G\times_H-) \colon H\sA \rtarr G\sA$. Then 
\[\bK(G\times_H-)^* \colon \mathbf{Pre}(G\sA,\sS) \rtarr \mathbf{Pre}(H\sA,\sS)\]
gives a model for the restriction $G\sS \rtarr H\sS$.
\end{con}

\begin{con}[Induction]\label{Induction} Let $H\subset G$ be a subgroup. The spectrum-level induction functor $G_+ \sma_H - \colon H\sS \rtarr G\sS$ is left adjoint to restriction. Given the description of restriction provided in \autoref{Restriction}, it follows that induction can be described as the enriched Kan extension (as in \cite[Lemma~2.2]{GM0})
\[ \bK(G\times_H-)_! \colon \mathbf{Pre}(H\sA,\sS) \rtarr \mathbf{Pre}(G\sA,\sS)\]
along the spectral functor $\bK(G\times_H-)\colon H\sA \rtarr G\sA$.
\end{con}

\begin{con}[Geometric inflation along a quotient]\label{Inflation} Let $N \trianglelefteq G$ be a normal subgroup. Then passage to $N$-fixed points defines a functor $\Fix^N\colon G\sE \rtarr G\!/\!N\sE$. Note that since $\Fix^N(A)$ is a subset of $A$, the $G/N$-set $\Fix^N(A)$ inherits an ordering from that of $A$. Moreover, $\Fix^N$ preserves pullbacks and coproducts. It follows that $\Fix^N$ gives rise to a spectral functor $\bK(\Fix^N)\colon G\sA \rtarr G\!/\! N \sA$. Then
\[\bK(\Fix^N)^* \colon \mathbf{Pre}(G\!/\!N\sA,\sS) \rtarr \mathbf{Pre}(G\sA,\sS)\]
gives a model for the geometric inflation functor, whose image consists of $G$-spectra ``concentrated over $N$''. In the language of \cite[Section~VI.5]{MM}, this is the functor $X\mapsto \widetilde{E\mathcal{F}}[N] \sma \varepsilon^\# X$, where $\varepsilon\colon G\rtarr G/N$ is the quotient homomorphism and $\varepsilon^{\#}$ is left adjoint to the $N$-fixed point functor from $G$-spectra to  $G/N$-spectra.
\end{con}

\begin{con}[Geometric fixed points]\label{GeomFixed} Let $N \trianglelefteq G$ be a normal subgroup. Then the geometric $N$-fixed points functor is left adjoint to geometric inflation. Given the description of geometric inflation provided in \autoref{Inflation}, the enriched Kan extension (as in \cite[Lemma~2.2]{GM0})
\[\bK(\Fix^N)_! \colon \mathbf{Pre}(G\sA,\sS) \rtarr \mathbf{Pre}(G\!/\!N\sA,\sS)\]
gives a model for the geometric $N$-fixed points functor $\Phi^N\colon G\sS \rtarr G\!/\!N \sS$.

This construction extends to arbitrary subgroups as follows. For a subgroup $H\subset G$, the $H$-fixed points functor $\Fix^H\colon G\sE \rtarr \sE$ gives rise to a spectral functor $\bK(\Fix^H)\colon G\sA \rtarr \sA$, and the enriched Kan extension 
\[\bK(\Fix^H)_! \colon \mathbf{Pre}(G\sA,\sS) \rtarr \mathbf{Pre}(\sA,\sS)\]
gives a model for the geometric $H$-fixed points functor $\Phi^H\colon G\sS \rtarr  \sS$.
We leave it to the reader to verify that, in the case of a normal subgroup, the two versions agree after restricting from $G/N$-spectra to underlying spectra.
\end{con}

\begin{con}[Categorical fixed points]\label{FixedPoints}
There is an inclusion $\iota\colon \sE \hookrightarrow G\sE$ of the finite sets as the $G$-trivial finite $G$-sets. This functor preserves pullbacks and coproducts and therefore induces a spectral functor $\bK(\iota)\colon  \sA \hookrightarrow G\sA$.  As generalized equivariantly in \autoref{bigboy}, spectrally enriched presheaves on finite sets are determined by their value at a one-point set,
and
\[\bK(\iota)^* \colon \mathbf{Pre}(G\sA,\sS) \rtarr \mathbf{Pre}(\sA,\sS)\simeq \sS\]
gives a model for the (categorical) $G$-fixed points functor $(-)^G \colon G\sS \rtarr \sS$. For a subgroup $H\subset G$, the $H$-fixed points functor is given by first using the restriction functor of \autoref{Restriction} and then passing to fixed points.
\end{con}

\begin{con}[$G$-trivial $G$-spectra]\label{Constant}
Left adjoint to the $G$-fixed points functor is the trivial $G$-action functor.
Given the description of $G$-fixed points provided in \autoref{FixedPoints}, the enriched Kan extension (as in \cite[Lemma~2.2]{GM0})
\[\bK(\iota)_! \colon \sS\simeq \mathbf{Pre}(\sA,\sS) \rtarr \mathbf{Pre}(G\sA,\sS)\]
gives a model for the trivial $G$-spectrum functor $\varepsilon^\# \colon \sS \rtarr G\sS$ (using the notation of \cite[Section~VI.3]{MM}).
This functor describes the tensoring of $G$-spectra over nonequivariant spectra. We return to this in \autoref{TensorsandSmash}.
\end{con}

\subsection{Fixed point orbit functors}
\label{FixedOrbitFunctors}

We return to \autoref{reassure} and give a more precise formulation.  
We know from \autoref{FixedPoints} how to pass to $H$-fixed points for each $H$, but a more functorial perspective may be illuminating.  
Again let $\sO_G$ denote the orbit category of $G$.  For a $G$-spectrum $X$, passage to $H$-fixed point spectra for $H\subset G$ gives a functor $X^{\bullet}\colon \sO_G^{op}\rtarr \sS$. Recall \autoref{AllOrb}.
By definition, $\GDOrb$ is the image of the composition $j$ of $\SI^\infty_{G,+} \colon \sO_G \rtarr G\sS$ with our bifibrant replacement functor.    Pulling back along $j$ defines a functor
\[ G\sS \xrightarrow{\bU}  \mathbf{Pre}(\GDOrb,\sS) \xrightarrow{j^*} \mathrm{Pre}(\sO_G,\sS),\]   
where the target denotes ordinary (i.e. unenriched) presheaves.
On the other hand, we have the functor $k\colon \sO_G \rtarr G\sE$ that associates to a 
map of finite $G$-sets its graph, considered as a span. This gives rise to a functor $\sO_G \rtarr G\sA$, 
which we also denote by $k$. Now pullback along $k$ gives a functor 
\[  \mathbf{Pre}(G\sA,\sS) \xrightarrow{k^*}  \mathrm{Pre}(\sO_G,\sS).\]

\begin{cor}\label{reassure2}  The zigzag of equivalences of \autoref{MAINNew1}
identifies the composition $j^* \circ \bU$ with $k^*$ up to equivalence.
\end{cor}

\subsection{Tensors with spectra and smash products}
\label{TensorsandSmash}

There is another visible identification.  The category $G\sS$ and our presheaf categories are
$\sS$-complete, so that they have tensors and cotensors over $\sS$ (see \cite[\S5.1]{GM0}). It is 
formal that the left adjoint of an $\sS$-adjunction preserves tensors and the right adjoint preserves 
cotensors.  A quick chase of our zigzag of Quillen $\sS$-equivalences gives the following conclusion.

\begin{prop} For $G$-spectra $Y$ and spectra $X$, if $Y$ corresponds to a presheaf $\sP Y$ 
under our zigzag of weak equivalences, then the tensor $Y\odot X$ corresponds to the tensor 
$\sP Y\odot X$.
\end{prop}

\begin{rem}[Smash products]
\label{Monoidal}
We have not described the behavior of smash
products under the equivalences of \autoref{MAINNew1}. On the presheaf side, one would expect to 
use Day convolution to describe the smash product, starting from the cartesian product of finite $G$-sets. 
Indeed, this is the approach taken in \cite{CMNN}, where a symmetric monoidal version of 
\autoref{MAINNew1} is given. We warn the reader, however, of two notable differences in their approach. First, in the approach 
of \cite{CMNN}, the functor from $G$-spectra to presheaves is a {\it left} adjoint, so that their right adjoint plays the role of our $\bT$ 
in \autoref{MAINOneCor}. Secondly, they produce a monoidal functor on the category of $G$-spectra by using \cite[Theorem~A.2]{CMNN}
that the category of $G$-spectra can be obtained as a monoidal category from the category of based $G$-spaces by inverting
smash products with representation spheres. 
\end{rem}

\begin{rem}
\label{MonoidalSketch}
We here give a sketch of an approach to a monoidal version of \autoref{MAINNew1}. Starting from an enriched symmetric monoidal structure on $\GDAll$,  Day convolution provides a symmetric monoidal structure on our category of spectral presheaves, and \autoref{start} can be promoted to a monoidal Quillen equivalence, as in \cite[Theorem~4.3]{ABS}. 
It then remains to equip the spectral category $G\sA$ with an enriched monoidal structure and
promote \autoref{MAINNew2} to a zig-zag of {\em monoidal} weak equivalences. 

However, there are several difficulties with this approach. First, starting with the enriched monoidal structure on 
$\GDAll$, it is clear what to do on objects, since they are in bijective correspondence with finite $G$-sets. Namely, again employing the notation of \autoref{PROOF}, the objects are of the form $R\bA = R \SI^\infty_G A_+$, 
and we define a product $\otimes$ on $\GDAll$ by  letting $R\bA \otimes R\bB$  be  
$R(\bA \sma \bB )\iso  R \SI^\infty_G(A\times B)_+$.

We next require a map of spectra
\begin{equation}\label{otimesFunctor} 
F( R\bA, R\bB) \sma F(R\bC, R\bD) \rtarr F(R\bA \otimes R\bC, R\bB \otimes R\bD). 
\end{equation}
If we had a strong monoidal fibrant replacement functor $R$, this would provide isomorphisms $R\bA \sma R\bB \iso R(\bA \sma \bB) = R\bA \otimes R\bB$. These could then be combined with the map
\[ F( R\bA, R\bB) \sma F(R\bC, R\bD) \rtarr F(R\bA \sma R\bC, R\bB\sma R\bD) \]
to obtain the  map \eqref{otimesFunctor}.
However, absent such a strong monoidal functor $R$, we do not see a way to define \eqref{otimesFunctor}.  
We shall say a bit more fibrant replacement in \autoref{Sdetails}. One way around this problem would be to rework the entire theory with orthogonal $G$-spectra replaced by the $S_G$-modules of the equivariant version \cite{MM} of \cite{EKMM}.  Since all $S_G$-modules are fibrant, that would get around this problem; some relevant details are discussed in \autoref{Burn} and \autoref{Zdetails}.

Another problem is that it is not straightforward to equip $G\sA$ with an enriched monoidal structure. Again, it is clear what to do on objects.  The machine developed in \cite{GMMOMain} does convert the product functors 
\begin{equation}
\label{GEProduct}
\xymatrix{
G\sE(B\times A) \times G\sE(D\times C) \ar[r]^-{\times} & G\sE(B\times A \times D \times C) \ar[r]^{\iso} & G\sE(B \times D \times A \times C)
}
\end{equation}
of \autoref{product}
 to morphisms of spectra
\[\bK G\sE(B\times A) \sma \bK G\sE(D\times C) \rtarr  \bK G\sE(B \times D \times A \times C)\]
 However, recall from \autoref{KGE} that the morphism spectra of $G\sA$ are defined using $G\Eone$ rather than $G\sE$, so some care is required to handle that change.  A little more seriously, even if we ignore the difference between $G\sE$ and $G\Eone$, the functors \eqref{GEProduct} do not give a strict 
 2-functor $G\Eone \times G\Eone \xrightarrow{\times} G\Eone$
since the evident diagram relating products to composition (of 1-cells) only commutes up to isomorphism.  We have not pursued this idea further, but we do not believe that the difficulties to this approach are insurmountable.
 \end{rem}
 
\section{Atiyah duality for finite $G$-sets}\label{SecAt}

It is illuminating to see that we can come very close to constructing an alternative model for 
the spectrally enriched category $\GDAll$ just by applying the suspension $G$-spectrum functor $\SI^{\infty}_G$ 
to the category of {based finite $G$-sets} and $G$-maps and then passing to $G$-fixed points. This is based on a close
inspection of classical Atiyah duality specialized to finite $G$-sets.  However, it depends on working in the
alternative category $G\sZ$ of $S_G$-modules \cites{EKMM, MM} rather than in the category $G\sS$ of 
orthogonal $G$-spectra.   Because every object of $G\sZ$ is fibrant and its suspension $G$-spectra are easily
understood, it is considerably more convenient than $G\sS$ for comparison with space level constructions.  This leads us to a 
variant, \autoref{MAINNew3}, of \autoref{MAINMAIN} that does not invoke infinite loop space theory. {It is more topological and less categorical,
and it best captures the geometric intuition behind our results.}
It is also more elementary.

\subsection{The categories $G\sZ$, $\GDZAll$, and $\DZAll$}\label{Burn}

Relevant background about $G\sZ$ appears in \autoref{Zdetails},
and we just give a minimum of 
notation here. We alert the reader to one non-standard notation.  We indicate the tensor of a based $G$-space  $X$ and a $G$-spectrum $E$ by $X\odot E = \Sigma^\infty_G X \sma E$.  Similarly, we later denote the tensor of a nonequivariant spectrum  $D$ and a $G$-spectrum $E$ by $D\odot E$.   

In analogy with \autoref{MAINOneCor}, we have the following specialization of the same general result,  \cite[Theorem~1.36]{GM0},  about stable model categories. It is discussed in \autoref{PresheafModels}.

\begin{thm}\label{MAINOneCorToo}  Let $\GDZAll$ be the full $\sZ$-subcategory of $G\sZ$ whose 
objects are cofibrant approximations of the suspension $G$-spectra $\SI^{\infty}_G(A_+)$,
where $A$ runs through the finite $G$-sets. Then there is an enriched Quillen adjunction
\[\xymatrix@1{\mathbf{Pre}(\GDZAll,\sZ) \ar@<.4ex>[r]^-\bT & G\sZ \ar@<.4ex>[l]^-{\bU}, }\]
and it is a Quillen equivalence.
\end{thm}

{We must be explicit about cofibrant approximation here.} The construction of the category $G\sZ$ of $S_G$-modules starts from the 
Lewis-May category $G\sS\!p$ of $G$-spectra, and $S_G$-modules are $G$-spectra with 
additional structure.  We have an elementary suspension $G$-spectrum functor
$\SI^{\infty}_G \colon G\sT\rtarr G\sS\! p.$
As we recall in \autoref{Zdetails}, a suspension $G$-spectrum has a canonical $S_G$-module 
structure, so that we may view $\SI^{\infty}_G$ as a functor $G\sT\rtarr G\sZ$. Moreover, with codomain 
$G\sZ$, this becomes a strong
symmetric monoidal functor.  However, the $\SI^{\infty}_G X$ are not cofibrant.  As explained in \autoref{Zdetails} below,  there is a left Quillen equivalence $\bF\colon G\sS\! p\rtarr G\sZ$ such that the composite  
$\SIG = \bF\com \SI^{\infty}_G$  takes based $G$-CW complexes $X$, such as $A_+$ for a 
finite $G$-set $A$, to cofibrant $S_G$-modules. Therefore $\SIG$ may be viewed
as a cofibrant replacement functor for $\SI^{\infty}_G$.  In particular, we write $\Sph = \SIG S^0$ 
and have a cofibrant approximation $\ga\colon \Sph\rtarr S_G$ of the unit object $S_G$. Moreover,
the cofibrant approximation $\SIG(A_+)$ is isomorphic over $\SI^{\infty}_G(A_+)$ to $\Sph\sma \SI^{\infty}_G(A_+)$.

As before, we consider finite $G$-sets $A$, $B$, and $C$, but we now agree to write  
\[ \bA = \SIG A_+,\ \ \ \bB = \SIG B_+, \ \ \ \text{and}
\ \ \ \bC = \SIG C_+.  \]
These are bifibrant objects of $G\sZ$ and we let $\GDZAll$ {and $\DZAll$} 
be the full subcategories of $G\sZ$ and 
$\sZ_G$ whose objects are the $S_G$-modules $\bA$, where $A$ runs over the 
finite $G$-sets.  Then $\DZAll$ is 
enriched in $G\sZ$ and 
$\GDZAll = (\DZAll)^G$ is enriched in the category $\sZ$ of $S$-modules.
The functor $\SIG$ is almost strong symmetric monoidal.  Precisely, by \autoref{OK} below,
 there is a natural isomorphism
\begin{equation}\label{yucky} 
\bA \sma \bB \iso \Sph\sma \SIG (A\times B)_+ 
\end{equation}
with appropriate coherence properties with respect to associativity and commutativity.
Since $S_G$ is the unit for the smash product, we can compose with 
\[ \ga\sma\id\colon \Sph\sma \SIG (A\times B)_+ \rtarr \SIG (A\sma B)_+\]
to give a pairing as if $\SIG$ were a lax symmetric monoidal functor.
However, the map $\ga\colon \Sph\rtarr S_G$ points the wrong way for 
the unit map of such a functor.

\subsection{Space level Atiyah duality for finite $G$-sets}\label{Atiyah}

To lift the self-duality of $Ho\sD_{\text{All}}$ to obtain a new model for $\GDZAll$,
we need representatives in $G\sZ$ for the maps 
\[  \et \colon S_G \rtarr \bA\sma \bA\ \ \ \text{and} \ \ \ 
\epz\colon \bA\sma \bA \rtarr S_G \]
in $\text{Ho}G\sZ$ that express the duality there.  The map $\epz$ is induced 
from the elementary map $\epz$ of \autoref{epsilon}. The observation that it plays 
a key role in Atiyah duality seems to be new.  The
definition of $\eta$ requires desuspension by representation spheres.  

Let $A$ be a finite $G$-set and let $V = \bR[A]$ be the real representation 
generated by $A$, with its standard inner product, so that $|a|=1$ for $a\in A$.  
Since we are working on the space level, we may view $A_+\sma S^V$ 
as the wedge over $a\in A$ of the spaces (not $G$-spaces) 
$\{a\}_+\sma S^V$, with $G$ acting by $g(a,v) = (ga,gv)$. There is no 
such wedge decomposition after passage to $G$-spectra.

\begin{defn}\label{defnepz} Recall that 
$\epz \colon (A\times A)_+ \rtarr S^0$ is the $G$-map defined by 
$\epz(a,b) = \ast$ if $a\neq b$ and $\epz(a,a) = 1$.  Recall too that
$(A\times B)_+$ can be identified with $A_+\sma B_+$ and that the functor $\SIG$ 
is almost strong symmetric monoidal. We shall also write $\epz$ for the composite
map of $S_G$-modules
\begin{equation}\label{epzgood}
\xymatrix@1{ \bA\sma \bA \iso \Sph\sma\SIG(A\times A)_+\ar[rr]^-{\id\sma \SIG\epz} 
& &  \Sph\sma \Sph \ar[r]^-{\ga\sma\ga} & S_G\sma S_G\iso S_G,}
\end{equation}
where the first unlabeled isomorphism is an instance of (\ref{yucky}).
\end{defn}

\begin{defn}\label{defneta} Embed $A$ as the basis of the real representation $V=\bR[A]$. 
The normal bundle of the embedding is just $A\times V$, and its Thom complex 
is $A_+\sma S^V$. We obtain an explicit tubular embedding $\nu\colon A\times V\rtarr V$ by setting 
\[ \nu(a,v) = a+\tfrac{\rh(|v|)}{|v|}v, \]
where $\rh\colon [0,\infty)\rtarr [0, d)$ is a homeomorphism for some $ d <1/2$; $\nu$ 
is a $G$-map since $|gv|=|v|$ for all $g$ and $v$. Applying the Pontryagin-Thom construction, 
we obtain a $G$-map  $t\colon S^V\rtarr A_+\sma S^V$, which is an equivariant pinch map
\[ S^V\rtarr \wed_{a\in A}S^V \iso A_+\sma S^V.\]
To be more precise, after collapsing the
complement of the tubular embedding to a point, we use $\nu^{-1}$ to expand each small
homeomorphic copy of $S^V$ to the canonical full-sized one; explicitly, 
if $|w|< d$, then
\[ \nu^{-1}(a + w) = (a,\tfrac{\rh^{-1}(|w|)}{|w|}w). \]
The diagonal map on $A$ induces the Thom diagonal
$\DE\colon A_+\sma S^V\rtarr A_+\sma A_+ \sma S^V$, and we let
\begin{equation}\label{etaA} \et  = \et_A \colon S^V\rtarr A_+\sma A_+\sma S^V
\end{equation} 
be the composite $\DE\com t$. Explicitly, 
\begin{equation}\label{etaA2} 
\et(v) = \left\{ \begin{array}{ll}
(a,a, \tfrac{\rh^{-1}(|w|)}{|w|}w) 
 & \mbox{if $v = a + w$ where $a\in A$ and $|w|< d$} \\
 \ast & \mbox{otherwise.}  
 \end{array}
 \right. 
\end{equation} 
The negative sphere $G$-spectrum $S^{-V}$ in
$G\sS\!p$ is obtained by applying the left adjoint of the $V^{th}$-space functor
to $S^0$, and $S_G$ is isomorphic (on the point-set level) to $S^V\odot S^{-V}$ as is noted nonequivariantly in  \cite[I.4.2]{LMS}\footnote{The relevant display there has a typo, $\OM^{\infty}$ for $\SI^{\infty}$.}.
Taking the tensor of $\et$ with $S^{-V}$ we obtain 
a map of $G$-spectra
\begin{equation*}
S_G\iso S^V\odot S^{-V} \rtarr (A_+\sma A_+\sma S^V)\odot S^{-V} 
\iso (A_+\sma A_+) \odot S_G \iso \SI^{\infty}_G (A_+\sma A_+).
\end{equation*}
Applying the functor $\bF$ to this map and smashing with $\Sph$ on the left, we obtain the 
map denoted $\hat{\et}_A$ in the diagram
\begin{equation}\label{etabad}
\xymatrix@1{ 
S_G \iso S_G\sma S_G &  \Sph\sma \Sph \ar[l]_-{\ga\sma\ga} \ar[r]^-{\hat{\et}_A}
& \Sph\sma \SIG{(A\times A)_+}\iso \bA\sma \bA. } 
\end{equation}
\end{defn}

The following result is a reminder about space level Atiyah duality. The notion of a 
$V$-duality was defined and explained for smooth $G$-manifolds in \cite[\S III.5]{LMS}.
Essentially, this states that the space-level maps $\eta$ and $\epz$  make $A_+$ into a self-dual $G$-space, modulo inverting the $G$-space $S^V$.
While our maps are specified precisely on the point-set level, we now pass to the homotopy category.

\begin{prop}\label{AtiyahSpace} The maps 
\[ \et\colon S^V\rtarr A_+\sma A_+\sma S^V \  \ \text{and} \  \ 
\epz\sma \id\colon A_+\sma A_+\sma S^V\rtarr S^V \] 
specify a $V$-duality between $A_+$ and itself.
\end{prop}  

\begin{proof}
This could be proven from scratch by proving the required triangle identities, but in fact 
it is a special case of equivariant Atiyah duality for smooth $G$-manifolds, $A$ being a
$0$-dimensional example.  Our specification of $\et$ is a precise point-set level specialization of the description
of $\et$ for a general smooth $G$-manifold $M$ given in \cite[p. 152]{LMS}.   Similarly, we claim that 
our $\epz\sma \id$ is a precise point-set level specialization of 
the definition of $\epz$ for a general smooth $G$-manifold given there. Indeed, letting
$s$ be the zero section of the normal bundle $\nu$ of the embedding $A\subset \bR[A] = V$, 
we have the composite embedding
\[ \xymatrix@1{ 
A \ar[r]^-{\DE} & A\times A \ar[r]^-{s\times\mathrm{id}} & (A\times V)\times A 
\iso A\times A\times V.} \]
The normal bundle of this embedding is $A\times V$, and we may view 
\[\DE\times{\mathrm{id}}\colon A\times V \rtarr A\times A\times V\]
as giving a big tubular neighborhood. The Pontryagin-Thom map here is 
obtained by smashing the map $r\colon (A\times A)_+\rtarr A_+$ that sends
$(a,b)$ to $a$ if $a=b$ and to $\ast$ if $a\neq b$ with the identity
map of $S^V$. Composing with the map induced by the projection
$\pi\colon A_+\rtarr S^0$ that sends $a$ to $1$, this gives $\epz\sma\id$. 
We observed this factorization of $\epz$ in \autoref{epsilon}
and we have used it before, in the proof of \autoref{curious}.
\end{proof}

We obtain the spectrum level duality maps displayed in 
(\ref{epzgood}) and (\ref{etabad})  by tensoring with $S^{-V}$, applying the functor $\Sph\sma \bF$, and composing 
with $\ga$. 

\subsection{The weakly unital categories $G\sB$ and $\sB_G$}

Since the $G$-spectra $\bA$ are self-dual, $F_G(\bA,\bB)$ is naturally isomorphic to 
$\bB\sma \bA$ in $\text{Ho}G\sZ$, and the composition and unit
\begin{equation}\label{real} F_G(\bB,\bC) \sma F_G(\bA, \bB)
\rtarr F_G(\bA,\bC) \ \ \ \text{and} \ \ \ S_G \rtarr F_G(\bB,\bB) 
\end{equation}
can be expressed as maps
\begin{equation}\label{fake} \bC \sma \bB \sma \bB \sma \bA \rtarr \bC\sma \bA 
\ \ \ \text{and} \ \ \ S_G \rtarr \bA\sma \bA
\end{equation}
in $\text{Ho}G\sZ$. We want to understand these maps in terms
of duality in $G\sZ$, without use of infinite loop space theory.  However, since we are working
in $G\sZ$, we must take the isomorphisms (\ref{yucky}) and the cofibrant approximation 
$\ga\colon \mathbf{S}_G\rtarr S_G$ into account, and we cannot expect to have strict units. 
The notion of a weakly unital enriched category was introduced in \cite[\S3.5]{GM0} to 
formalize what we see here. 

Thus we shall construct a weakly unital $G\sZ$-category $\sB_G$, analogous to $\sA_G$, and 
compare it with $\DZAll$.  The $G$-fixed category $G\sB$ will be a weakly unital $\sZ$-category.
The objects of $\sB_G$ and $G\sB$ are the $S_G$-modules $\bA$ for finite $G$-sets $A$, as in \autoref{Burn}.
 The morphism 
$S_G$-modules of $\sB_G$ are $\sB_G(\bA,\bB) =\bB\sma \bA$.  Composition is given by the maps
\begin{equation}\label{compZ}
\id\sma\epz\sma\id\colon \bC \sma \bB \sma \bB \sma \bA \rtarr \bC\sma \bA,
\end{equation}
where $\epz$ is the map of (\ref{epzgood}); compare \autoref{curious}. 

As recalled in \autoref{duality}, the adjoint $\tilde{\epz}\colon \bA\rtarr D\bA = F_G(\bA,S_G)$ of 
$\epz$ is a stable equivalence, and we have the composite stable equivalence
\begin{equation}\label{deltaZ}
\de = \ze\com (\id\sma\tilde{\epz}) \colon \bB\sma \bA \rtarr \bB\sma D\bA \rtarr F_G(\bA,\bB).
\end{equation}
Formal properties of the adjunction ($\sma$,$F_G$) give the 
following commutative diagram in $G\sZ$, which uses $\de$ to
compare composition in $\sB_G$ with composition in $\DZAll$. 
\begin{equation}\label{compZZ} \xymatrix{
\bC \sma \bB \sma \bB \sma \bA \ar[rr]^-{\id\sma \epz\sma \id}
\ar[d]_{\mathrm{id}\sma\tilde{\epz}\sma\mathrm{id}\sma\tilde{\epz}}
& & \bC \sma \bA \ar[d]^{\mathrm{id}\sma \tilde{\epz}} \\
\bC \sma D\bB \sma \bB \sma D\bA 
\ar[rr]^-{\mathrm{id}\sma \epz \sma \mathrm{id}} \ar[d]_{\ze\sma\ze}
& & \bC \sma  D\bA \ar[d]^{\ze}\\
F_G(\bB,\bC) \sma F_G(\bA, \bB)\ar[rr]_-{\com} & &  F_G(\bA,\bC) }
\end{equation}
At the bottom, we do not know that the function $S_G$-modules or their smash product are cofibrant, 
but all objects at the top are cofibrant and thus bifibrant. In general, to compute the smash product 
of $G$-spectra $X$ and $Y$ in the homotopy category, we should take the smash product of cofibrant 
approximations $\bQ X$ and $\bQ Y$ of $X$ and $Y$. Since all objects of $G\sZ$ are fibrant, to compute a map 
$X\sma Y\rtarr Z$ in the homotopy category, we should represent it by a map $\bQ X\sma \bQ Y\rtarr \bQ Z$ and 
take its homotopy class. The diagram displays such a cofibrant approximation of the composition in $\DZAll$.

Specialized to our context of a category with self-dual objects, the definition (\cite[Definition~3.25]{GM0}) of a weakly unital $G\sZ$-category  requires, for each object $\bA$, a ``weak unit map" $\hat{\et}_A\colon \bQ S_G \rtarr \bA \sma \bA$ for some chosen cofibrant approximation $\ga\colon \bQ S_G \rtarr S_G$,  together with a weak equivalence 
$\hat{\xi}_\bA\colon \bA \xrightarrow{\simeq} \bA$ such that certain unit diagrams relating $\hat{\et}_A$, 
$\hat{\xi}_A$ and composition commute.  We are led by (\ref{etabad}) to choose our cofibrant approximation $\ga$ to be 
$ \ga\sma \ga\colon \Sph\sma\Sph\rtarr S_G\sma S_G\iso S_G,$ and to take  
$\hat{\et}_A\colon \Sph\sma\Sph\rtarr \bA \sma \bA$ to be the map displayed in \autoref{etabad}.  After composing with $\de\colon \bA \sma \bA \rtarr F_G(\bA,\bA)$, $\hat{\eta}_A$ is a representative in $G\sZ$ for the unit map $S_G \rtarr F_G(\bA,\bA)$ that exists in $\mathrm{Ho} G\sZ$. Finally, we specify the required equivalence 
$\hat{\xi}_A\colon \bA \xrightarrow{\simeq} \bA$.

\begin{defn}\label{defnxiA}  Let $V=\bR[A]$.
For $a\in A$, define $\xi_a\colon \{a\}_+\sma S^V\rtarr \{a\}_+\sma S^V$ by
\begin{equation}\label{xiA1} 
\xi_a(a,v) = \left\{ \begin{array}{ll}
(a,(\rh^{-1}(|w|)/|w| )w) 
 & \mbox{if $v= a+w$ and $|w|< d$} \\
 \ast & \mbox{otherwise,}  
 \end{array}
 \right. 
\end{equation} 
where $\rh$ is as in \autoref{defneta}. 
Then the wedge of the $\xi_a$ is a $G$-map 
\begin{equation}\label{xiA2} 
\xi_A \colon A_+\sma S^V \rtarr A_+\sma S^V;
\end{equation}  
$\xi_A$ is $G$-homotopic to the identity map of $A_+\sma S^V$ via the explicit $G$-homotopy
\[ h(a,v,t) =\left\{ \begin{array}{ll}
(a,v) & \mbox{if $t=0$ or $v=a$} \\
(a,(1-t)v + t(\rh^{-1}(t|w|)/|w|)w) & \mbox{if $v=a+w$ and $t|w| < d$} \\
\ast & \mbox{otherwise.}
\end{array} \right. \] 
{
Tensoring with $S^{-V}$ and using the natural isomorphisms 
\[ (X\sma S^V)\odot S^{-V}\iso X\odot S_G \iso \SI^{\infty}_G X \]
for based $G$-spaces $X$, we see that the space level $G$-equivalence $\xi_A$ induces a spectrum level $G$-equivalence 
$\hat{\xi}_A\colon \bA\rtarr\bA$.}
\end{defn}

It is a bit tedious to verify that our definitions make $\sB_G$ into a weakly unital $G\sZ$-category, in the sense specified in \cite[Definition~3.25]{GM0}. Here are the details.

With $\et_A$ as specified in (\ref{etaA}), easy and perhaps illuminating inspections show that the following 
unit diagrams already commute in $G\sT$, before passage to homotopy. In both, $A$ and $B$ are finite $G$-sets. 
In the first, $V= \bR[A]$. In the second, 
$W= \bR[B]$ and we move $S^W$ from the right to the left for clarity.
\[ \xymatrix{
B_+\sma A_+\sma S^V \ar[r]^-{\id\sma\et_A} \ar[d]_{\id\sma \xi_A}
&   B_+\sma A^3_+ \sma S^V \ar[dl]^-{\quad \id\sma\epz\sma\id}\\
B_+\sma A_+ \sma S^V & } 
\ \ \text{and} \ \ 
\xymatrix{
\ar[d]_{\xi_B \sma \id_A} S^W\sma B_+\sma A_+\ar[r]^-{\et_B\sma\id}
&  S^W\sma B^3_+\sma A_+  \ar[dl]^{\quad \id\sma\epz\sma\id}\\
S^W\sma B_+\sma A_+ } \]   

Tensoring with $S^{-V}$ { and $S^{-W}$} and using (\ref{yucky}) to pass to smash products of
$S_G$-modules, a little diagram chase shows that the previous pair of diagrams 
in $G\sT$ gives rise to the following pair of commutative diagrams in $G\sZ$.  
These express the unit laws for a weakly unital $G\sZ$-category $\sB_G$ \cite[Definition~3.25]{GM0} 
with objects the $\bA$ and composition as specified in (\ref{compZ}). 
Again, the cited unit laws 
allow us to start with any chosen cofibrant approximation $\ga\colon \bQ S_G \rtarr S_G$ of 
the unit $S_G$, and we were led by (\ref{etabad}) to choose our cofibrant approximation $\ga$ to be 
$ \ga\sma \ga\colon \Sph\sma\Sph\rtarr S_G\sma S_G\iso S_G.$
The space level diagrams  above induce the required spectrum level diagrams
\[ \xymatrix{
\ar[d]_{\id\sma\hat{\xi}_A\sma \ga} \bB\sma \bA \sma \bQ S_G \ar[r]^-{\id\sma\hat{\et}_A}
&  \bB\sma \bA\sma \bA\sma \bA  \ar[d]^{\com}\\
\bB\sma \bA \sma S_G \ar[r]_{\iso} & \bB\sma \bA } 
\ \ \text{and} \ \ 
\xymatrix{
\ar[d]_{\ga\sma \hat{\xi}_B\sma \id} \bQ S_G \sma  \bB\sma \bA\ar[r]^-{\hat{\et}_B\sma\id}
& \bB\sma \bB\sma \bB\sma \bA  \ar[d]^{\com}\\
S_G\sma \bB\sma \bA \ar[r]_{\iso}  & \bB\sma \bA. } \]

Taking $A= S^0$ in our second space level diagram and changing $B$ to $A$, we also obtain the following 
commutative diagrams in $G\sZ$, where the second diagram is adjoint to the first. 
\begin{equation}\label{unitZ}
 \xymatrix{
\ar[d]_{\ga\sma \hat{\xi}_A} \bQ S_G \sma \bA   \ar[r]^-{\hat{\et}_A\sma\id}
& \bA\sma \bA\sma \bA  \ar[d]^{\id\sma \epz}\\
S_G \sma \bA  \ar[r]_-{\iso} & \bA } 
\ \ \text{and} \ \ 
\xymatrix{
\bQ S_G \ar[d]_{\ga} \ar[r]^-{\hat{\et}_A} & \bA\sma \bA \ar[r]^-{\id\sma \tilde{\epz}} & \bA\sma D\bA  \ar[d]^{\ze}\\
S_G \ar[r]_-{\et} & F_G(\bA,\bA)  \ar[r]_-{F_G(\hat{\xi}_A,\id)} & F_G(\bA,\bA) } 
\end{equation}
Here $\et$ at the bottom left of the right diagram is adjoint to the identity map of $\bA$. 
 In effect, this uses $\de = \ze\com(\id\sma\tilde{\epz})$ to compare  the unit 
$S_G \xrightarrow{\et} F_G(\bA,\bA)$ in $\DZAll$ with the ``weak unit" $S_G \leftarrow \bQ S_G \rtarr  \bA \sma \bA$ in $\sB_G$.

\subsection{The category of presheaves with domain $G\sB$}\label{MAINAt}

The diagrams (\ref{compZZ}) and (\ref{unitZ}) show that the maps $\de\colon \bA\sma\bB \rtarr F_G(\bA,\bB)$ 
specify a map of weakly unital $G\sZ$-categories from the weakly unital $G\sZ$-category $\sB_G$
to the (unital) $G\sZ$-category $\DZAll$. Passing to $G$-fixed points, we obtain
a weakly unital $\sZ$-category $G\sB$ and a map $\de\colon G\sB\rtarr \GDZAll$
of weakly unital $\sZ$-categories.  Weakly unital presheaves and presheaf 
categories are defined in \cite[Definition~3.25]{GM0}.  By \cite[Remark~3.26]{GM0},
we obtain the same category of presheaves 
$\mathbf{Pre}(\GDZAll,\sZ)$
using unital or
weakly unital presheaves. Since $\de$ is an equivalence, we can adapt the 
methodology of \cite[\S2]{GM0} to complete the proof of the following theorem, using the details
relating the functor $\SIG$ to smash products from \autoref{Zdetails}.  Since we 
find the use of weakly unital categories unpleasant and our main result
\autoref{MAINNew1} more satisfactory, we shall leave the details to the interested reader. 

\begin{thm}\label{MAINNew3} The categories $\mathbf{Pre}(G\sB,\sZ)$ and $\mathbf{Pre}(\GDZAll,\sZ)$ are Quillen equivalent.
\end{thm}

\section{Appendix: Enriched model categories of $G$-spectra}\label{background}

The results in this section show how to model categories of $G$-spectra as categories of presheaves of 
spectra, where $G$ is any compact Lie group.  We specialize results of \cite{GM0} to provide and compare two such models.
More precisely, in  \autoref{PresheafModels} we establish Theorems \ref{MAINOneCor} and \ref{MAINOneCorToo},  which 
state that $G$-spectra can be modeled as presheaves of spectra in both the orthogonal and $S$-module contexts.
In \autoref{OldMain}, we compare these two presheaf models. In sections \ref{Sdetails} and \ref{Zdetails} we discuss suspension 
spectra for orthogonal spectra and $S$-modules, respectively, in order to be precise about the domain categories for our presheaves.
We shall rely on \cites{EKMM, LMS, MM, MMSS} for definitions of the relevant categories.

\subsection{Presheaf models for categories of $G$-spectra}\label{PresheafModels}

We focus on two categories of $G$-spectra treated in detail in \cite{MM}.  We have the closed 
symmetric monoidal category $\sS$ of nonequivariant orthogonal 
spectra \cite{MMSS}. Its function spectra are denoted $F(X,Y)$.  We also have the closed symmetric 
monoidal category $G\sS$ of orthogonal $G$-spectra for a fixed $G$-universe $U$ \cite{MM}.
Its function $G$-spectra are denoted $F_G(X,Y)$. 
In contrast to the previous sections, in this subsection and the next we allow $G$-spectra to be indexed over any $G$-universe.   The homotopy type of $F_G(X,Y)$ very much  depends on the choice of universe.
Then $G\sS$ is enriched over $\sS$
via the $G$-fixed point spectra $F_G(X,Y)^G$. In terms of the general context of \cite{GM0}, we are
taking $\sV= \sS$ and $\sM = G\sS$. We have stable model structures on $\sS$ and $G\sS$ 
\cites{MM, MMSS}.  

Then 
\cite[Theorem~1.36]{GM0} specializes to give
\autoref{MAINOneCor}.  It  also gives the following more general result, in which $G$ can be a compact Lie group and $G$-spectra can be indexed on any universe.  (See also \cite[Example~3.4(i)]{SS}).

\begin{thm}\label{MAINOne} Let  $G\sD_S$ be the full $\sS$-subcategory of $G\sS$ whose objects are fibrant approximations of the suspension $G$-spectra $\SI^{\infty} X_+$ for all $X$ in any set $S$ of compact $G$-spaces that contains $G/H$ for at least one $H$ in each conjugacy class of closed subgroups of $G$. Then there is an enriched Quillen adjunction
\[\xymatrix@1{\mathbf{Pre}(G\sD_S,\sS) \ar@<.4ex>[r]^-\bT & G\sS \ar@<.4ex>[l]^-{\bU}, }\]
and it is a Quillen equivalence.  If $S\subset T$ are as prescribed and 
$$\bR\colon \mathbf{Pre}(G\sD_T ,\sS) \rtarr \mathbf{Pre}(G\sD_S,\sS)$$
is the restriction along the inclusion $G\sD_S \rtarr G\sD_T$,  then $\bR\com \bU_T = \bU_S$ and therefore $R$ induces an equivalence of presheaf homotopy categories.
\end{thm}

\begin{rem}  Adapting our work for finite groups to incomplete universes would requre us to use incomplete Mackey functors  and to reconcile the conflict between needing to use all orbits $G/H$ to obtain generators for $\text{Ho} G\sS$ and needing to use only those orbits $G/H$ that embed in the given universe to have self-duality of orbits, which is vital to our theory but irrelevant to \autoref{MAINOne}.
\end{rem}

We have a second specialization of \cite[Theorem~1.36]{GM0}.  We have the closed symmetric monoidal
category $\sZ$ of nonequivariant $S$-modules \cite{EKMM}.\footnote{The notation $\sS$ is short 
for $\sI\sS$ and the notation $\sZ$ is short for $\sM_S$ in the original sources; as a silly 
mnemonic device, $\sZ$ stands for the $Z$ in the middle of Elmendorf-KriZ-Mandell-May.} Its function 
spectra are again denoted $F(X,Y)$.  We also have the closed symmetric monoidal category $G\sZ$ of 
$S_G$-modules (for a fixed $G$-universe $U$ as above) \cite{MM}. Its function $G$-spectra are denoted 
$F_G(X,Y)$.  Then $G\sZ$ is enriched over $\sZ$ via the $G$-fixed point spectra $F_G(X,Y)^G$.  We are
taking $\sV= \sZ$ and $\sM = G\sZ$. We have stable model structures on $\sZ$ and $G\sZ$ 
\cites{EKMM, MM}. Again, \cite[Theorem~1.36]{GM0} specializes to give  \autoref{MAINOneCorToo}. It also gives the following more general result, in which $G$ can be a compact Lie group and $G$-spectra can be indexed on any universe.

\begin{thm}\label{MAINOneToo} Let  $G\sD^{\sZ}_S$ be the full $\sS$-subcategory of $G\sZ$ whose objects are cofibrant approximations of the suspension $G$-spectra $\SI^{\infty} X_+$ for all $X$ in any set $S$ of compact $G$-spaces that contains $G/H$ for at least one $H$ in each conjugacy class of closed subgroups of $G$. Then there is an enriched Quillen adjunction
\[\xymatrix@1{\mathbf{Pre}(G\sD^{\sZ}_S,\sZ) \ar@<.4ex>[r]^-\bT & G\sZ \ar@<.4ex>[l]^-{\bU}, }\]
and it is a Quillen equivalence.  If $S\subset T$ are as prescribed and 
$$\bR\colon \mathbf{Pre}(G\sD_T ,\sZ) \rtarr \mathbf{Pre}(G\sD_S,\sZ)$$
is the restriction along the inclusion $G\sD^{\sZ}_S \rtarr G\sD^{\sZ}_T$,  then $\bR\com \bU_T = \bU_S$ and therefore $R$ induces an equivalence of presheaf homotopy categories.
\end{thm}

\begin{rem}\label{bigboy}  When $G$ is finite, we focus on the set $S = Orb$ of all orbit $G$-sets $G/H$ and the set $T=All$ of all finite $G$-sets. Here we can obtain an inverse equivalence to $\bR$ 
by sending a presheaf defined on $S$ to an {\em additive} presheaf defined on $T$, where additivity requires a presheaf that  sends disjoint unions in $T$ to finite products in $G\sS$ or in $G\sZ$.  Thus an interpretation of the equivalence of presheaves on $\GDOrb$ with presheaves on $\GDAll$ is that presheaves on $\GDAll$ are equivalent to additive presheaves.  The intuition is that the spectral enrichment builds in additivity, just as functors enriched over abelian groups automatically preserve coproducts. 
\end{rem}

Homotopically, Theorems \ref{MAINOne} and \ref{MAINOneToo}  are essentially the same result 
since $G\sS$ and $G\sZ$ are Quillen equivalent. On the point set level they are quite different, and they 
have different virtues and defects. 

We say just a bit about the proofs of these theorems.  By \cite[Theorem~4.32]{GM0},
the presheaf categories used in them are well-behaved model categories.
The acyclicity condition there holds in \autoref{MAINOne} because $\sS$ satisfies the monoid axiom, 
by \cite[12.5]{MMSS}. It holds in \autoref{MAINOneToo} by use of the 
``Cofibration Hypothesis'' of \cite[p. 146]{EKMM}, which also holds equivariantly.
The orbit $G$-spectra give compact generating sets in both $\text{Ho}(G\sS)$ and 
$\text{Ho}(G\sZ)$. We require bifibrant representatives. In \autoref{MAINOne}, 
the orbit $G$-spectra are cofibrant, and fibrant approximation makes them bifibrant.

By contrast, in \autoref{MAINOneToo}, all $S_G$-modules are fibrant, and cofibrant approximation makes them bifibrant. 
Here cofibrant approximation is given by a well understood left adjoint that very nearly preserves smash products, 
as we shall explain in \autoref{Zdetails}.

Technically, \cite[Theorem~1.36]{GM0} requires {\em either} that the unit object of the enriching category $\sV$ be 
cofibrant {\em or} that every object in $\sV$ be fibrant.   The first hypothesis holds in $\sS$ and the second holds 
in $\sZ$. It is impossible to have both of these conditions in the same symmetric monoidal model category for the 
stable homotopy category \cites{Lewis, MayRant}. That is a key reason that both of these results are of interest.

\subsection{Comparison of presheaf models of $G$-spectra}\label{OldMain}
Theorems \ref{MAINOne} and \ref{MAINOneToo} are related by the following result, which is
\cite[IV.1.1]{MM}; the nonequivariant special case is \cite[I.1.1]{MM}. In this result,
$G\sS$ is given its positive stable model structure from \cite{MM} and is denoted 
$G\sS_{pos}$ to indicate the distinction; in that model structure, the sphere 
$G$-spectrum in $G\sS$, like the sphere $G$-spectrum in $G\sZ$, is not cofibrant.
In \cite{MM}, the result is proven for genuine $G$-spectra for compact Lie groups $G$.  
For arbitrary topological groups $G$, the same proof applies to classical $G$-spectra, that is $G$-spectra  indexed on a universe with trivial $G$-action.

\begin{thm}\label{NNsharp}  There is a Quillen
equivalence
\[\xymatrix@1
{G\sS_{pos} \ar@<.4ex>[r]^-\bN & G\sZ \ar@<.4ex>[l]^-{\bN^{\#}}. }\]
The functor $\bN$ is strong symmetric monoidal, hence $\bN^{\#}$ is
lax symmetric monoidal.
\end{thm}

The identity functor is a left Quillen equivalence $G\sS_{pos} \rtarr G\sS$. 
Therefore Theorems \ref{MAINOne}, \ref{MAINOneToo}, and \ref{NNsharp}, 
have the following immediate consequence. 

\begin{cor}\label{MAINOneTwo} The categories $\mathbf{Pre}(\GDOrb,\sS)$ and 
$\mathbf{Pre}(\GDZOrb,\sZ)$ 
are Quillen equivalent.  More precisely, there are left Quillen equivalences
\[  \mathbf{Pre}(\GDOrb,\sS) \rtarr G\sS \ltarr G\sS_{pos} \rtarr G\sZ \ltarr \mathbf{Pre}(\GDZOrb,\sZ). \]
\end{cor} 

In fact, we can compare the $\sS$-category $\GDOrb$ with the $\sZ$-category 
$\GDZOrb$ via the right adjoint $\bN^{\#}$.  The adjunction 
\[\xymatrix@1
{G\sS_{pos} \ar@<.4ex>[r]^-\bN & G\sZ \ar@<.4ex>[l]^-{\bN^{\#}} }\]
is tensored over the adjunction
\[\xymatrix@1
{\sS_{pos} \ar@<.4ex>[r]^-\bN & \sZ \ar@<.4ex>[l]^-{\bN^{\#}} }\]
in the sense of \cite[Definition~3.20]{GM0}.  Indeed, since $G\sS$ is a bicomplete $\sS$-category, it is 
tensored over $\sS$. While a more explicit definition is easy enough,  for a spectrum $X$ and $G$-spectrum $Y$ we can define the $G$-spectrum
$Y\odot X$ to be $Y\sma i_*\epz^*X$, where $i_*\epz^*\colon \sS\rtarr G\sS$ is the 
change of group and universe functor associated to $\epz\colon G\rtarr e$ that assigns a 
genuine $G$-spectrum to a nonequivariant spectrum. The same is true with $\sS$ replaced by $\sZ$. 
These functors are discussed in both contexts and compared in \cite{MM}. Results there (see \cite[IV.1.1]{MM}) 
imply that 
\[ \bN Y\odot \bN X \iso \bN(Y\odot X), \]
which is the defining condition for a tensored adjunction.
Now \cite[Corollary~3.24]{GM0} gives that the $\sS$-category $\bN^{\#}\GDZOrb$ is quasi-equivalent 
to $\GDOrb$. Using \cite[Remark~2.15 and Theorem~3.17]{GM0}, this implies a direct proof of the Quillen 
equivalence of \autoref{MAINOneTwo}. Therefore Theorems \ref{MAINOne} and \ref{MAINOneToo} 
are equivalent: each implies the other.

We reiterate the generality: the results above do {\em not} require $G$ to be finite.
In that generality, we do not know how to simplify the description of the domain 
category $\GDOrb$ to transform it into a weakly equivalent $\sS$-category or $\sZ$-category 
that is intuitive and perhaps even familiar, something accessible to study independent of 
knowledge of the category of $G$-spectra that we seek to understand. 
Our main theorem shows how to do just that when $G$ is finite.

\subsection{Suspension spectra and fibrant replacement functors in $G\sS$}\label{Sdetails}

We here give some observations relevant to understanding the category $\GDOrb$ of \autoref{MAINOne}.
From now on, the group $G$ is finite and the universe is complete unless otherwise specified.

For an inner product space $V$ and a based $G$-space $X$, the $V^{th}$ space of  the orthogonal $G$-spectrum
$\SI^{\infty}_GX$ is $X\sma S^V$. {The functor $\SI^{\infty}_G$, also denoted $F_0$, 
is left adjoint to the zero$^{th}$ space functor $(-)_0\colon G\sS\rtarr G\sT$.  Nonequivariantly, 
it is part of \cite[1.8]{MMSS} that for based spaces $X$ and $Y$, $F_0 X\sma F_0Y$ 
is naturally isomorphic to $F_0(X\sma Y)$. The categorical proof of that result 
in \cite[\S21]{MMSS} applies equally well equivariantly to give the following  result. 

\begin{prop} The functor $\SI^{\infty}_G\colon G\sT\rtarr G\sS$ is strong
symmetric monoidal.
\end{prop}

Therefore the zero$^{th}$ space functor is
lax symmetric monoidal, but of course that functor is not homotopically meaningful 
except on objects that are fibrant in the stable model structure. There is no
known fibrant replacement functor in that model structure that is well-behaved
with respect to smash products.}
Recall from \autoref{MonoidalSketch} that the existence of a monoidal fibrant replacement functor is relevant to a monoidal version of our main result.

Although it is less useful for our purposes, we point out two different constructions of monoidal fibrant replacement functors in the {\em positive} stable model structure.
The first is immediate from  \autoref{NNsharp} but does not appear in the literature.
  
\begin{prop}\label{fibrant}  The unit $\et\colon E\rtarr \bN^{\#}\bN E$ 
of the adjunction between $G\sS$ and $G\sZ$ specifies a lax monoidal fibrant 
replacement functor on cofibrant objects for the positive stable model structure  $G\sS_{\mathrm{pos}}$.
\end{prop}

\begin{rem} Nonequivariantly, Kro \cite{Kro} has given a different lax monoidal
positive fibrant replacement functor for orthogonal spectra.  His construction does
not require restriction to cofibrant objects.  Parenthetically, as he notes, it does not apply to symmetric spectra.  However, by \cite[3.3]{MMSS},  
the unit $E\rtarr \bN^{\sharp}\bU\bP\bN E$ of the composite of the 
adjunction $(\bP,\bU)$ between symmetric and orthogonal spectra and
the adjunction $(\bN,\bN^{\sharp})$ gives a lax monoidal positive fibrant
replacement functor for symmetric spectra.
\end{rem}

Unfortunately the restriction to the positive model structure in \autoref{fibrant} is necessary, 
and the only fibrant approximation functor we know of for use with the stable model structure employed in \autoref{MAINOne} is that given by the small object argument.  The point is that the suspension
$G$-spectra $\SI^{\infty}_G (G/H_+)$ are cofibrant but not positive cofibrant.

Nonequivariantly, a homotopically meaningful version of the adjunction
$(\SI^{\infty},\OM^{\infty})$ has been worked out for symmetric spectra 
by Sagave and Schlichtkrull \cite{SaSc} and for symmetric and orthogonal spectra
by Lind \cite{Lind}, who compares his constructions with the adjunction 
$(\SI^{\infty},\OM^{\infty})$ in $\sS\!p$ (see below)
and with its analogue for $\sZ$. This generalizes to the equivariant 
context, although details have not been written down.

\subsection{Suspension spectra and smash products in $G\sZ$}\label{Zdetails}
We here give some observations relevant to understanding the category $\GDZOrb$ of \autoref{MAINOneToo}.
In particular, we give properties of cofibrant approximations of suspension spectra that are used in \autoref{SecAt}. 
For more information, see \cite[XXIV]{EHCT}, \cite[\S IV.2]{MM}, and the nonequivariant precursor \cite{EKMM}.

We have a category $G\sP$ of (coordinate-free)-prespectra. Its objects $Y$
are based $G$-spaces $Y(V)$ and based $G$-maps $Y(V)\sma S^W\rtarr Y(W-V)$
for $V\subset W$.  Here $V$ and $W$ are sub inner product spaces of a 
$G$-universe $U$.  A $G$-spectrum is a $G$-prespectrum $Y$ whose adjoint
$G$-maps $Y(V)\rtarr \OM^{W-V}Y(W)$ are homeomorphisms. The (Lewis-May) category $G\sS\!p$ of 
$G$-spectra is the full subcategory of $G$-spectra in $G\sP$.   The suspension $G$-prespectrum
functor $\PI$ sends a based $G$-space $X$ to $\{ X\sma S^V\}$.  There is a left adjoint
spectrification functor $L\colon G\sP \rtarr G\sS\!p$, and the suspension $G$-spectrum 
functor $\SI_G^{\infty}\colon G\sT\rtarr G\sS\! p$ is $L\com\PI$. Explicitly, let 
\[ Q_G X = \colim \OM^V\SI^V X, \] 
where $V$ runs over the
finite dimensional subspaces of a complete $G$-universe $U$.  Then the $V^{th}$
$G$-space of $\SI_G^{\infty} X$ is $Q_G \SI^V X$.

All objects of $G\sS\!p$ are fibrant,
and the zero$^{th}$ space functor $\OM^{\infty}_G\colon G\sS\! p\rtarr G\sT$ is now 
homotopically meaningful. For a based $G$-CW complex $X$ (with based attaching maps), 
$\SI^{\infty}_GX$ is cofibrant in $G\sS\! p$. In particular, the sphere $G$-spectrum 
$S_G = \SI^{\infty}_GS^0$ is cofibrant.  Since $G$ is a compact Lie group,
the orbits $G/H$ are $G$-CW complexes, hence the $\SI_G^{\infty}(G/H_+)$ are cofibrant.
However, $G\sS\!p$ is not symmetric monoidal under the smash product. The implicit trade off here is 
intrinsic to the mathematics, as was explained by Lewis \cite{Lewis}; see \cite{MayRant} for a more recent discussion.  

We summarize some constructions in \cite{EKMM} that work in exactly the same fashion
equivariantly as nonequivariantly.  We
have the $G$-space $\sL(j)$ of linear isometries $U^j\rtarr U$, with $G$ acting by
conjugation. These spaces form an $E_{\infty}$ $G$-operad when $U$ is complete.  The 
$G$-monoid $\sL(1)$ gives rise to a monad $\bL$ on $G\sS\!p$. Its algebras are called 
$\bL$-spectra, and we have the category $G\sS\! p[\bL]$ of $\bL$-spectra.  It has a smash 
product $\sma_{\sL}$ which is associative and commutative but not unital. The action map 
$\xi\colon \bL Y\rtarr Y$ of an $\bL$-spectrum $Y$ is a stable equivalence. 

Suspension $G$-spectra are naturally $\bL$-spectra. In particular, the sphere 
$G$-spectrum $S_G$ is an $\bL$-spectrum.  There is a natural stable equivalence 
$$\la\colon S_G\sma_{\sL} Y\rtarr Y$$ 
for $\bL$-spectra $Y$.  The $S_G$-modules are those 
$Y$ for which $\la$ is an isomorphism, and they are the objects of $G\sZ$. All suspension 
$G$-spectra are $S_G$-modules, and so are all $\bL$-spectra of the form 
$S_G\sma_{\sL} Y$.  The smash product $\sma$ on $S_G$-modules is just the restriction 
of the smash product $\sma_{\sL}$, and it gives $G\sZ$ its symmetric monoidal structure.  

We have a sequence of Quillen left adjoints
\[\xymatrix{ G\sT \ar[r]^-{\SI^{\infty}_G} & G\sS\!p \ar[r]^-{\bL} 
& G\sS\! p[\bL] \ar[r]^-{\bJ} & G\sZ, } \]
where $\bL X$ is the free $\bL$-spectrum generated by a $G$-spectrum $X$ and 
$\bJ Y = S_G \sma_{\sL}Y$ is the $S_G$-module generated by an $\bL$-spectrum $Y$.  
We let $\bF = \bJ\bL$; then $\bL$, $\bJ$, and $\bF$ are Quillen equivalences.
The composite $\ga = \xi\circ \la\colon \bF Y\rtarr Y$ is a stable equivalence
for any $\bL$-spectrum $Y$. We have defined $\SIG$ to be the composite functor 
$\bF\SI^{\infty}_G$, and we have the natural stable equivalence of $S_G$-modules 
$\ga\colon \SIG X\rtarr \SI^{\infty}_G X$. 

The tensor $Y\odot X$ of a $G$-prespectrum and a based $G$-space $X$ has $V^{th}$
$G$-space $Y(V)\sma X$.  When $Y$ is a $G$-spectrum, the $G$-spectrum $Y\odot X$ is 
$L(\ell Y\odot X)$, where $\ell Y$ is the underlying $G$-prespectrum
of $Y$ \cite[I.3.1]{LMS}.
Tensors in $G\sS\!p[\bL]$ and $G\sZ$ are inherited from those in $G\sS\!p$. All of our 
left adjoints are enriched in $\sT$ and preserve tensors. This leads to the following 
relationship between $\sma$ and $\SIG$.

\begin{prop}\label{OK} 
For based $G$-spaces $X$ and $Y$, there are natural isomorphisms
\[ \SIG X \sma \SIG Y \iso (\Sph\sma \Sph)\odot (X\sma Y) 
\iso  \Sph\sma \SIG (X\sma Y). \]
\end{prop}
\begin{proof} We have 
$\SI^{\infty}_G X \iso S_G \odot X$ and therefore
\[ \SIG X = \bF\SI^{\infty}_GX \iso \bF(S_G\odot X) \iso (\bF S_G)\odot X =\Sph\odot X.\]
We also have   
\[ (\Sph\odot X)\sma (\Sph\odot Y)\iso (\Sph\sma\Sph)\odot (X\sma Y) \]
and the conclusion follows.
\end{proof}

\section{Appendix: Whiskering $G\sE$ to obtain strict unit $1$-cells}\label{Whisker} 

The bicategory $G\sE$ of \autoref{span1a} narrowly misses being a strict 2-category, and we whisker the unit $1$-cells to obtain a strict $2$-category $\GEone$.\footnote{We thank Ang\'{e}lica Osorno for help with the material in this section.}
Before focusing on specifics we give an elementary general definition.

\begin{defn}\label{def:whisker}
 For a category $\sD$ with a privileged object $\DE$, define the whiskering $\sD'$ of $\sD$ at $\DE$ by adjoining a new object
$I$ and an isomorphism $\ze\colon I\rtarr \DE$.  We have the inclusion $i\colon \sD \rtarr \sD'$, and we define a retraction functor
$r\colon \sD'\rtarr \sD$ by $r(I) = \DE$ and $r(\ze) = \id_{\DE}$.  Thus $r\com i = \text{Id}_{\sD}$ and the isomorphism $\ze$ on the 
object $I$ together with the identity map on all other objects of $\sD'$ defines a natural isomorphism $\text{Id}_{\sD'} \rtarr i\com r$.
If $\sD$ is a $G$-category and $\DE$ is $G$-fixed, then $\sD'$ is a $G$-category with $I$ and $\ze$ fixed by $G$, and then $\sD$ and $\sD'$ are $G$-equivalent.
\end{defn}

The whiskered category $\GEone$ ``enriched in permutative categories'' and the whiskered $G$-category $\EGone$ ``enriched in permutative $G$-categories''
are defined to have the same objects, or $0$-cells, as $G\sE$ and $\sE_G$, namely the finite $G$-sets $A$ in both cases.

\begin{defn}\label{span1b}
If $A\neq B$ or if $|A|\leq 1$ and $A=B$, we define $\GEone(A,B)$ to be the permutative category $G\sE(A,B)$.
For each $A$ of cardinality at least $2$, we define 
\[\GEone(A,A) = G\sE(A, A)'\!,\]
where the whiskering is performed at the 1-cell $\Delta_A$.
We denote the adjoined $1$-cell by $I_A$ and the adjoined isomorphism $2$-cell by $\ze_A\colon I_A \rtarr \Delta_A$.  We specify a permutative structure on  
$\GEone(A,A)$ by setting
\[ E{\amalg}F = \left\{ \begin{array}{ll} {I_A} & \text{if}\ (E,F)=({I_A},\emptyset ) \ \text{or}\ (\emptyset ,{I_A}) \\ 
i\big(r(E)\amalg r(F) \big) 
& \text{otherwise.}
\end{array}\right.
\]
We have denoted the monoidal product as $\amalg$ since the product in $G\sE(A\times A)$ is given by the disjoint union of spans.
As the only $2$-cell in $\GEone(A,A)$ with source or target $\emptyset $ is $\id_\emptyset $, this product extends uniquely to a functor. Since the retraction 
\[r\colon \GEone(A,A) \rtarr G\sE(A, A)\]
is strict monoidal 
and an equivalence of categories, the symmetry isomorphism 
$\gamma\colon \amalg \iso \amalg\tau$ on $G\sE(A,A)$ lifts uniquely to a symmetry isomorphism $\gamma\colon \amalg \iso \amalg\tau$ 
on $\GEone(A,A)$.
Observe that the inclusion $i\colon G\sE(A,A) \to \GEone(A,A)$ is strict monoidal.

To extend composition to functors
\[ \xymatrix@1{ \GEone(B,C)\times \GEone(A,B) \ar[r]^-{\circ} & \GEone(A,C) \\} \]
we declare ${I_A}$ to be a strict 2-sided unit. It remains to define composition with a 2-cell with source or target ${I_A}$.  Since every such 2-cell factors through $\zeta_A$ and composition with $\DE_A$ is already defined, it suffices to define composition with 
$\zeta_A$. Since $\Delta_A$ is  a strict right unit, for a span $B \longleftarrow E\rtarr A$, abbreviated $E$, we may define $E\circ \zeta_A: E\circ {I_A} \rtarr E\circ \Delta_A$ to be the identity $2$-cell $\id_E$. We define $\zeta_B\circ E: {I_B}\circ E \rtarr \Delta_B \circ E$ to be $\ell_{B,E}^{-1}$,  where $\ell_{B,E}$ is the $2$-cell defined in (\ref{laDe}).
\end{defn}

\begin{rem}
In \cite{BO}, and also in a previous version of this article, a different strictification of $G\sE$ was proposed, namely just redefining composition with $\Delta_A$ to force this to be a unit $1$-cell. Unfortunately, this breaks associativity, since the 1-cell $\Delta_A$ is decomposable under composition if $|A|\geq 2$. 
\end{rem}

We have a precisely analogous definition on the level of $G$-categories, obtaining a strict $2$-category $\EGone$ from $\sE_G$. 

\begin{defn}\label{span2b}
If $A\neq B$ or if $|A|\leq 1$ and $A=B$, we define $\EGone(A,B)$ to be the permutative $G$-category $\sE_G(A,B)$.
For each $A$ of cardinality at least $2$, we define 
\[\EGone(A,A) = \sE_G(A, A)'.\]
We denote the adjoined $1$-cell by $I_A$ and the adjoined isomorphism $2$-cell by $\ze_A$.  We specify a $G$-permutative structure on  $\EGone(A,A)$ by setting 
\[ \tha(\mu;E_1,\dots,E_n) = \left\{ \begin{array}{ll} {I_A} & \text{if}\ E_i = I_A \ \text{and}\ E_j= \emptyset \ \text{for all}\ j\neq i \\
 \tha(\mu;r(E_1),\dots, r(E_n)) & \text{otherwise.}
\end{array}\right. 
\] 
Observe  that the inclusion $i\colon \sE_G(A,A) \to \EGone(A,A)$ is a map of $\sP_G$-algebras.

To extend composition to a functor
\[ \xymatrix@1{ \EGone(B,C)\times \EGone(A,B) \ar[r]^-{\circ} & \EGone(A,C),} \]
we declare the object $I_A\in \EGone(A,A)$ to be a strict $2$-sided unit. 
We define composition with a 2-cell whose source or target is of the form $I_A$ exactly as in  
\autoref{span1b}, except that to define $\ze_B\circ E$ we now use  the $\ell_{B,E}$ defined in \eqref{laDeG}.
\end{defn}

\bibliographystyle{amsalpha}

\end{document}